\documentclass[11pt,a4paper]{amsart}
\usepackage{amssymb,amsmath}
\usepackage{mathrsfs}
\usepackage[alphabetic]{amsrefs}
\usepackage{graphicx,color}
\usepackage[utf8]{inputenc}
\usepackage{amsthm}
\usepackage{latexsym}
\usepackage{slashed}
\usepackage[all]{xy}
\usepackage{hyperref}
\usepackage[mathscr]{eucal}
\usepackage{tikz}
\usepackage{mathtools}
\usepackage{accents}

\def\theequation{\thesection.\arabic{equation}}
\makeatletter
\@addtoreset{equation}{section}
\makeatother

\makeatletter
\newcommand{\eqnum}{\refstepcounter{equation}\textup{\tagform@{\theequation}}}
\makeatother

\newcounter{copy}
\makeatletter
\renewcommand{\thecopy}{\ifnum0=\c@section\arabic{copy}\else\thesection.\arabic{copy}'\fi}
\makeatother

\BibSpec{book}{%
    +{}  {\PrintPrimary}                {transition}
    +{.} { \textit}                     {title}
    +{.} { }                            {part}
    +{:} { \textit}                     {subtitle}
    +{,} { \PrintEdition}               {edition}
    +{}  { \PrintEditorsB}              {editor}
    +{,} { \PrintTranslatorsC}          {translator}
    +{,} { \PrintContributions}         {contribution}
    +{,} { }                            {series}
    +{,} { \voltext}                    {volume}
    +{,} { }                            {publisher}
    +{,} { }                            {organization}
    +{,} { }                            {address}
    +{,} { }                            {status}
    +{,} { \PrintDOI}                   {doi}
    +{,} { \PrintISBNs}                 {isbn}
    +{}  { \parenthesize}               {language}
    +{}  { \PrintTranslation}           {translation}
    +{;} { \PrintReprint}               {reprint}
    +{,} { \PrintDate}                  {date}
    +{.} { }                            {note}
    +{.} {}                             {transition}
}
\BibSpec{article}{%
    +{}  {\PrintAuthors}                {author}
    +{,} { \textit}                     {title}
    +{.} { }                            {part}
    +{:} { \textit}                     {subtitle}
    +{,} { \PrintContributions}         {contribution}
    +{.} { \PrintPartials}              {partial}
    +{,} { }                            {journal}
    +{}  { \textbf}                     {volume}
    +{}  { \PrintDatePV}                {date}
    +{,} { \issuetext}                  {number}
    +{,} { \eprintpages}                {pages}
    +{,} { }                            {status}
    +{,} { \PrintDOI}                   {doi}
    +{}  { \parenthesize}               {language}
    +{}  { \PrintTranslation}           {translation}
    +{;} { \PrintReprint}               {reprint}
    +{.} { }                            {note}
    +{,} { \eprint}                     {eprint}
    +{.} {}                             {transition}
}
\BibSpec{collection.article}{%
    +{}  {\PrintAuthors}                {author}
    +{,} { \textit}                     {title}
    +{.} { }                            {part}
    +{:} { \textit}                     {subtitle}
    +{,} { \PrintContributions}         {contribution}
    +{,} { \PrintConference}            {conference}
    +{}  {\PrintBook}                   {book}
    +{,} { }                            {booktitle}
    +{,} { \PrintDateB}                 {date}
    +{,} { pp.~}                        {pages}
    +{,} { }                            {publisher}
    +{,} { }                            {organization}
    +{,} { }                            {address}
    +{,} { }                            {status}
    +{,} { \PrintDOI}                   {doi}
    +{,} { \eprint}        {eprint}
    +{}  { \parenthesize}               {language}
    +{}  { \PrintTranslation}           {translation}
    +{;} { \PrintReprint}               {reprint}
    +{.} { }                            {note}
    +{.} {}                             {transition}
}

\BibSpec{misc}{
  +{}{\PrintAuthors}  {author}
  +{,}{ \textit}     {title}
  +{}{ (}             {date}
  +{),}{ }             {note}
  +{.}{}              {transition}
}

\theoremstyle{definition}
\newtheorem{defn}[equation]{Definition}

\theoremstyle{plain}
\newtheorem{thm}[equation]{Theorem}
\newtheorem{prp}[equation]{Proposition}
\newtheorem{lem}[equation]{Lemma}
\newtheorem{cor}[equation]{Corollary}

\theoremstyle{remark}
\newtheorem{rmk}[equation]{Remark}

\newcommand{\bB}{\mathbb{B}}
\newcommand{\bC}{\mathbb{C}}

\newcommand{\bF}{\mathbb{F}}

\newcommand{\bK}{\mathbb{K}}

\newcommand{\bM}{\mathbb{M}}
\newcommand{\bN}{\mathbb{N}}

\newcommand{\bT}{\mathbb{T}}

\newcommand{\cA}{\mathcal{A}}

\newcommand{\cC}{\mathcal{C}}

\newcommand{\cG}{\mathcal{G}}
\newcommand{\cH}{\mathcal{H}}

\newcommand{\cP}{\mathcal{P}}

\newcommand{\cR}{\mathcal{R}}

\newcommand{\cU}{\mathcal{U}}
\newcommand{\cV}{\mathcal{V}}
\newcommand{\cW}{\mathcal{W}}

\newcommand{\fe}{\mathfrak{e}}

\newcommand{\fu}{\mathfrak{u}}
\newcommand{\fv}{\mathfrak{v}}

\newcommand{\be}{\mathbf{e}}

\newcommand{\bu}{\mathbf{u}}
\newcommand{\bv}{\mathbf{v}}
\newcommand{\bw}{\mathbf{w}}

\newcommand{\id}{\mathrm{id}}

\DeclareMathOperator{\K}{\mathrm{K}}

\DeclareMathOperator{\End}{\mathrm{End}}
\DeclareMathOperator{\Hom}{\mathrm{Hom}}

\DeclareMathOperator{\Ad}{\mathrm{Ad}}
\DeclareMathOperator{\diag}{\mathrm{diag}}

\newcommand{\Bdl}{\mathrm{Bdl}}
\newcommand{\qRep}{\mathrm{qRep}}

\allowdisplaybreaks[3]

\title[Almost flat relative vector bundles]{Almost flat relative vector bundles and the almost monodromy correspondence}
\author{Yosuke KUBOTA}
\address{iTHEMS Research Program, RIKEN, 2-1 Hirosawa, Wako, Saitama 351-0198, Japan}
\email{yosuke.kubota@riken.jp}

\date{}

\begin{document}
\begin{abstract}
In this paper we introduce the notion of almost flatness for (stably) relative bundles on a pair of topological spaces and investigate basic properties of it. First, we show that almost flatness of topological and smooth sense are equivalent. This provides a construction of an almost flat stably relative bundle by using the enlargeability of manifolds. Second, we show the almost monodromy correspondence, that is, a correspondence between almost flat (stably) relative bundles and (stably) relative quasi-representations of the fundamental group.  
\end{abstract}

\maketitle
\tableofcontents

\section{Introduction}
The notion of \emph{almost flat bundle} provides a geometric perspective on the higher index theory.
It was introduced by Connes--Gromov--Moscovici~\cite{MR1042862} for the purpose of proving the Novikov conjecture for a large class of groups. 
The original definition is given in terms of curvature of vector bundles, and hence requires a smooth manifold structure for the base space. Another definition of almost flat bundle is given in \cite[Section 2]{MR2179351}, which make sense for bundles on simplicial complexes. The equivalence of these two definitions is studied in \cite{mathGT160707820}.

Its central concept is the \emph{almost monodromy correspondence}, that is, the correspondence between almost flat bundles and quasi-representations of the fundamental group. 
This almost one-to-one correspondence has been studied in various contexts such as \cite{MR1065438,MR1865957,MR3275029,mathOA150306170}. 
It plays an important role in the work of Hanke--Schick \cite{MR2259056,MR2353861}, which bridges the C*-algebraic and geometric approaches to the Novikov conjecture and the Gromov--Lawson--Rosenberg conjecture.

The aim of paper is to consider a similar problem for manifolds with boundary. For a pair of topological spaces $(X,Y)$, we introduce the notion of almost flatness for representatives of the relative $\K^0$-group $\K^0(X,Y)$. Our definition is inspired from the one suggested in \cite{MR1389019} and \cite{MR3101798}, but slightly different (a major difference is to treat stably relative bundles instead of relative bundles). 

There are two main conclusions of this paper. The first, Theorem \ref{thm:equiv}, is the comparison of topological and smooth almost flat bundles, a relative analogue of the result of \cite{mathGT160707820}.
This theorem has an application to the index theory of (area-)enlargeable manifolds.
Gromov--Lawson~\cite{MR720933} and Hanke--Schick~\cite{MR2353861} constructs an almost flat bundle of Hilbert C*-modules with non-trivial index on an enlargeable spin manifold. 
In this paper we consider a relative counterpart of this idea for a Riemannian manifold $M$ with boundary $\partial M$ such that the complete Riemannian manifold $M_\infty := M \sqcup _{\partial M} \partial M \times [0,\infty)$ is area-enlargeable. 
We construct a stably relative bundle of Hilbert C*-modules on $(M,\partial M)$ with non-trivial index pairing by applying the construction of Gromov--Lawson and Hanke--Schick (Theorem \ref{thm:enlarge}).

The second, Theorem \ref{prp:rel1to1}, is the almost monodromy correspondence in the relative setting. For a pair $(\Gamma, \Lambda)$ of discrete groups with a homomorphism $\phi \colon \Lambda \to \Gamma$, we introduce the notion of (stably) relative quasi-representation as two quasi-representations on $\Gamma$ whose composition with $\phi $ are stably unitary equivalent up to small $\varepsilon>0$. 
Following the work of Carri\'{o}n and Dadarlat \cite{mathOA150306170}, we establish an almost monodromy correspondence between almost flat relative bundles and relative quasi-representations of the pair of fundamental groups.
This correspondence plays an important role in the paper \cite{Kubota2}, which bridges the index pairing with almost flat stably relative bundles and Chang--Weinberger--Yu relative higher index. 
In particular, the almost flat stably relative bundle constructed in Theorem \ref{thm:enlarge} is used in \cite[Section 3.2]{Kubota2} to show the non-vanishing of the Chang--Weinberger--Yu relative higher index through the almost monodromy correspondence.

In this paper we consider not only relative vector bundles (or Karoubi triples) but also its refinement, \emph{stably relative vector bundles}, as a representative of the relative $\K^0$-group and sometimes compare them. A stably relative vector bundle on $(X,Y)$ is a pair of vector bundles on $X$ identified by a stable unitary isomorphism on $Y$ (for a more precise definition, see Definition \ref{defn:stb}). There are two reasons to consider stably relative bundles. 
The first is related with the enlargeable manifolds. What is obtained from the enlargeability of $M_\infty$ is not a relative but a stably relative bundle. 
The second is related with the almost monodromy correspondence. As is pointed out in Remark \ref{rmk:unstable}, relative quasi-representation of $(\Gamma , \Lambda )$ is the same thing as that of $(\Gamma, \phi(\Lambda))$. That is, relative quasi-representation does not capture any information of $\ker \phi $.

This paper is organized as follows. In Section \ref{section:2}, we introduce the notion of stably relative bundle and show that it represents an element of the relative $\K^0$-group. In Section \ref{section:3}, we introduce the definition of the almost-flatness for stably relative bundles. 
In Section \ref{section:4}, we compare the topological and smooth almost flatness and applies this to enlargeable manifolds. In Section \ref{section:4.5}, we apply the result of Section \ref{section:4} to a construction of an almost flat sequence of stably relative bundles on a enlargeable manifold with boundary.
In Section \ref{section:5}, we define the relative analogue of group quasi-representations and shows the almost monodromy correspondence. 

Throughout this paper, we treat bundles of (finitely generated projective) Hilbert C*-modules. This general treatment is useful for generalizing Hanke--Schick theorem for a generalized notion of enlargeability introduced in \cite{MR2353861} by using infinite covers.

\subsection*{Acknowledgment}
The author would like to thank Yoshiyasu Fukumoto for introducing him to this topic. This work was supported by RIKEN iTHEMS Program.

\section{Relative and stably relative bundles}\label{section:2}
In this section we introduce the definition of stably relative vector bundles and bundles of Hilbert C*-modules as a representative of relative $\K^0$-group.  
Throughout this section $A$ denotes a unital C*-algebra and $P, Q$ denote finitely generated projective Hilbert $A$-modules.

Let $(X,Y)$ be a pair of compact Hausdorff spaces. The relative K-group $\K^0(X,Y)$ is defined as the Grothendieck construction of the monoid of equivalence classes of triples $(E_1,E_2,u)$, where $E_1$ and $E_2$ are vector bundles on $X$ and $u$ is an isomorphism $E_1|_Y \to E_2|_Y$ (\cite[Chapter II, 2.29]{MR2458205}). In this paper we call such triple a \emph{relative vector bundle}. Now we modify this description of the group $\K^0(X,Y)$. For a unital C*-algebra $A$, we define the relative $\K^0$-group with coefficient in $A$ by $\K^0(X,Y;A) :=\K_0(C_0(X^\circ) \otimes A)$, where $X^\circ$ denotes the interior $X \setminus Y$.
\begin{defn}\label{defn:stb}
A \emph{stably relative bundle} on $(X,Y)$ with the typical fiber $(P,Q)$ is a quadruple $(E_1,E_2, E_0, u)$, where $E_1$ and $E_2$ are $P$-bundles on $X$, $E_0$ is a $Q$-bundle on $Y$ and $u$ is a unitary bundle isomorphism $E_1|_Y \oplus E_0 \to E_2|_Y \oplus E_0$. 
\end{defn}
A stably relative bundle of Hilbert $\bC$-modules with the typical fiber $(\bC^n, \bC^m)$ is simply called a stably relative vector bundle of rank $(n,m)$.

We say that stably relative bundles $(E_1 ,E_2 , E_0 , u)$ and $(E_1', E_2' , E_0' , u')$ are isomorphic if there are unitary isomorphisms $U_i \colon E_i \to E_i'$ for $i=0,1,2$ such that $\diag (U_2|_Y , U_0)u=u'\diag (U_1|_Y,U_0)$. 
Let $\Bdl_s(X,Y ; A)$ denote the set of isomorphism classes of stably relative bundles of finitely generated projective Hilbert $A$-modules. We consider the equivalence relation $\sim$ on $\Bdl_s(X,Y ; A)$ generated by 
\begin{itemize}\label{item:equiv}
\item $(E_1,E_2,E_0,u) \sim (E_1',E_2',E_0',u')$ if they are homotopic, that is, there is a stably relative vector bundle $(\tilde{E}_1, \tilde{E}_2, \tilde{E}_0, \tilde{u})$ on $(X[0,1] , Y[0,1])$ whose restriction to $(X \times \{ 0 \} , Y \times \{ 0 \})$ and $(X \times \{ 1 \} , Y \times \{ 1 \})$ are isomorphic to $(E_1,E_2,E_0,u) $ and $(E_1',E_2',E_0',u')$ respectively,  
\item $(E_1 , E_2, E_0 , u) \sim (0,0,E_1|_Y \oplus E_0,(v|_Y \oplus 1_{E_0})^*u)$ if $v$ is a unitary isomorphism from $E_1$ to $E_2$, and
\item $(0,0,E_0, 1_{E_0}) \sim 0$.
\end{itemize}
The summation $[E_1, E_2, E_0 , u] + [E_1', E_2', E_0' ,u']:=[E_1 \oplus E_1' , E_2 \oplus E_2', E_0 \oplus E_0' , u \oplus u'  ]$ makes the set $\Bdl_s(X,Y; A)/\sim$ into an abelian monoid. Moreover, $[E_1, E_2, E_0, u]$ has the inverse $[E_2, E_1, E_0 , u^*]$.

\begin{lem}
The group $\Bdl_s (X,Y ; A)/\sim $ is isomorphic to the relative $\K^0$-group $\K^0(X,Y ; A)$.
\end{lem}
\begin{proof}
In the proof, we write as $\bar{\K}^0(X,Y;A):=\Bdl _s(X,Y;A)/\sim $. Let $\rho \colon (C_0(X^\circ ) \otimes A)^+ \to \bC$ denote the quotient.
We define the map $\kappa \colon \K^0(X,Y;A) \to \bar{\K}^0(X,Y,A)$ by  
\[ \kappa ([p]-[1_n])= [p(A^N_X), A^n_X , 0 , 1_n]\]
for a projection $p \in \bM_N((C_0(X) \otimes A)^+)$ with $ \rho (p)=1_n$. 

For a compact space $X$, let $\K^*(X;A):=\K_*(C(X) \otimes A)$. Let $i^* \colon C(X) \to C(Y)$ denote the restriction and let $j \colon C_0(X^\circ) \to C(X)$ denote the inclusion. Consider the homomorphisms
\begin{align*}
\bar{\partial} \colon &\K^1(Y; A) \to \bar{\K}^0(X,Y;A), && \bar{\partial} [u] = [0, 0, A_Y^n, u], \\ 
\bar{j}_* \colon &\bar{\K}^0(X,Y;A) \to \K^0(X; A), && \bar{j}_*[E_1, E_2, E_0, u] = [E_1] - [E_2].
\end{align*}
Actually, the equivalence relation $\sim$ is defined in such a way that $\bar{\partial } $ and $\bar{j}$ are well-defined and the second row of the commutative diagram 
\[ \xymatrix{
\K^1(X; A) \ar[r]^{i^*} \ar@{=}[d] & \K^1(Y;A) \ar[r]^\partial \ar@{=}[d] & \K^0(X,Y;A) \ar[r]^{j_*} \ar[d]^\kappa  & \K^0(Y;A) \ar[r]^{i^*} \ar@{=}[d] & \K^0(X;A)  \ar@{=}[d] \\
\K^1(X; A) \ar[r]^{i^*} & \K^1(Y;A) \ar[r]^{\bar{\partial}} & \bar{\K}^0(X,Y;A) \ar[r]^{\bar{j}_*} & \K^0(X;A) \ar[r]^{i^*} & \K^0(Y;A) 
}\] 
is exact (for the exactness at $\K^1(Y;A)$,  note that $[0,0,A_Y^n, 1]=[A_X^n, A_X^n, 0 , 1] = [0, 0, A_Y^n, u]$ if $u \in \mathrm{U}(C(Y) \otimes A \otimes  \bM_n)$ is extended to a unitary in $C(X) \otimes A \otimes \bM_n$). Now the lemma follows from the five lemma.
\end{proof}

\section{Almost flatness for (stably) relative bundles}\label{section:3}
In this section we introduce the notion of $\varepsilon$-flatness for stably relative bundles of Hilbert $A$-modules.
Let us recall the definition of almost flat bundle on a topological space introduced in \cite{MR2179351}.  
\begin{defn}
Let $X$ be a locally compact space with a finite open cover $\cU:=\{ U_\mu \}_{\mu \in I}$. For a finitely generated Hilbert $A$-module $P$, a $\mathrm{U}(P)$-valued \v{C}ech $1$-cocycle $\bv=\{ v_{\mu \nu}\}_{\mu,\nu \in I}$ on $\cU$ is an \emph{$(\varepsilon , \cU)$-flat bundle} on $X$ with the typical fiber $P$ if $\|  v_{\mu\nu}(x)-v_{\mu\nu}(y)\| <\varepsilon$ for any $x,y \in U_{\mu\nu}:=U_\mu \cap U_\nu$.
\end{defn}
We write $\Bdl_P^{\varepsilon , \cU}(X)$ for the set of $(\varepsilon , \cU)$-flat bundles with the typical fiber $P$. For $\bv \in \Bdl_P^{\varepsilon , \cU}(X)$, we write $E_\bv$ for the underlying $P$-bundle.
\begin{rmk}\label{rmk:bundle}
For the latter use we realize the bundle $E_\bv$ as a subbundle of the trivial bundle $X \times P^n$. Let $ \{ \eta_\mu \}_{\mu \in I}$ be a family of positive continuous functions on $X$ such that $\sum_{\mu \in I} \eta_\mu ^2 =1 $ and let $e_{\mu \nu} \in \bM_{I}$ denotes the matrix element, i.e., $e_{\mu \nu} e_\sigma = \delta_{\nu , \sigma} e_\mu$ where $e _\mu$ is the standard basis of $\bC^I \cong \Hom (\bC, \bC^I)$. Let
\begin{align*}
p_{\bv}(x)&:= \sum_{\mu , \nu} \eta _\mu(x) \eta_\nu(x) v_{\mu \nu}(x) \otimes e_{\mu\nu} \in C(X) \otimes \bB(P) \otimes \bM_I, \\
\psi_\mu^\bv (x) &:= \sum_{\nu } \eta_\nu (x) v_{\nu \mu} (x)  \otimes e_\nu \in C_b(U_\mu) \otimes \bB(P) \otimes \bC^I.
\end{align*}
Here we regard $\psi_\mu^\bv(x)$ as a bounded operator between Hilbert $A$-modules $ P $ and $P \otimes \bC^I$ and consider its adjoint $\psi _\mu^\bv (x)^*=\sum \eta_\nu(x)v_{\nu\mu}(x)^* \otimes e_\nu^*$, where $\{ e_\nu^*\}_{\nu \in I} \subset \Hom(\bC^I,\bC)$ is the dual basis of $\{e_\nu\}$.  
Then we have $p_\bv(x) \psi_\mu ^{\bv}(x)=\psi_\mu(x)$ for $x \in U_\mu$ and  $\psi_\mu^{\bv}(x)^*\psi_{\nu}^\bv(x) = v_{\mu \nu}(x)$ for $x\in U_{\mu \nu}$. That is, $p_\bv$ is a projection with the support $E_\bv$ and $\psi_\mu ^\bv$ is a local trivialization of $E_\bv$.
\end{rmk}
\begin{defn}
For two $(\varepsilon , \cU)$-flat bundles $\bv_1=\{ v_{\mu\nu}^1\}$ and $\bv_2=\{ v_{\mu\nu}^2\}$, a \emph{morphism of $(\varepsilon , \cU)$-flat bundles} is a family of unitaries $\bu=\{ u_\mu \} _{\mu \in I} \in \mathrm{U}(P)^I$ such that 
\[ \sup _{\mu, \nu \in I} \sup_{x \in U_{\mu\nu}} \| u_\mu v_{\mu \nu}^{1}(x) u_\nu^* -v_{\mu\nu}^2(x) \|  < \varepsilon . \]
We write $\Hom _\varepsilon (\bv_1 , \bv_2)$ for the set of morphisms of $\varepsilon$-flat bundles. Moreover, for $\bu \in \Hom _\varepsilon (\bv_1 ,   \bv_2 )$ and $\delta>0$, let
\[ \cG_\delta (\bu):= \Big\{ \{ \bar{u}_\mu \colon U_\mu \to \mathrm{U}(P) \}_{\mu \in I}  \mid \begin{array}{l}\bar{u}_\mu(x) v_{\mu \nu}^1(x) \bar{u}_\nu(x)^* = v_{\mu \nu}^2(x), \\  \| \bar{u}_{\mu}(x) - u_\mu  \| < \delta \end{array} \Big\}_{\textstyle .} \]
\end{defn}
For $\bar{u} \in \cG_{\delta}(\bu)$, we use the same symbol $\bar{u}$ for the induced unitary isomorphism $\bar{u} \colon E_{\bv_1} \to E_{\bv_2}$.

\begin{lem}\label{lem:gauge}
There is a constant $C_1=C_1(\cU)>0$ depending only on the open cover $\cU$ such that the following hold. Let $0 < \varepsilon < (3C_1)^{-1}$, $\bv_1 , \bv_2 \in \Bdl_P^{\varepsilon , \cU}(X)$ and $\bu \in \Hom _{\varepsilon }( \bv_1 , \bv_2 )$.
\begin{enumerate}
\item The set $\cG_{C_1 \varepsilon} (\bu)$ is non-empty.
\item The inclusion $\cG_{ C_1 \varepsilon} (\bu) \to \cG_{3 C_1 \varepsilon}(\bu)$ is homotopic to a constant map.
\end{enumerate}
\end{lem}
\begin{proof}
By replacing $\bv_1$ with $\bu \cdot \bv_1:= \{u_\mu v_{\mu \nu}^1 u_\nu^* \}_{\mu, \nu \in I}$, we may assume that $u_\mu =1$ for any $\mu \in I$, that is, $\| v^1_{\mu\nu}(x) - v^2_{\mu \nu}(x)\|<\varepsilon$. Let $p_\bv$ and $\psi_{\mu}^\bv$ be as in Remark \ref{rmk:bundle}.

Set $C_1:=|I|^2+1$ (then $|I|^2\varepsilon < 1/3$). By the triangle inequality, we have
\begin{align} \| p_{\bv_1} - p_{\bv_2} \| \leq \sup _{x \in X} \sum _{\mu, \nu}\eta_\mu (x)\eta_\nu (x) \| v_{\mu\nu}^1(x)-v_{\mu\nu}^2(x) \| < |I|^2 \varepsilon \label{form:pdiff} \end{align}
and hence
\begin{align} \| p_{\bv_1}p_{\bv_2}p_{\bv_1} - p_{\bv_1} \| = \| p_{\bv_1}(p_{\bv_1} - p_{\bv_2})p_{\bv_1} \| < |I|^2\varepsilon. \label{form:ppp} \end{align}
Let us regard $p_{\bv_1}p_{\bv_2}p_{\bv_1}$ as an element of the corner C*-algebra $p_{\bv_1} (C(X) \otimes \bB(P) \otimes \bM_I) p_{\bv_1}$. Then the above inequality implies that
\[ \sigma (p_{\bv_1}p_{\bv_2}p_{\bv_1} ) \subset  [1-|I|^2\varepsilon, 1+|I|^2\varepsilon] \subset  [2/3,4/3]\] 
and especially $p_{\bv_1}p_{\bv_2}p_{\bv_1}$ is invertible.

Now we consider the polar decomposition of the bounded operator
\[ p_{\bv_2}p_{\bv_1} \colon p_{\bv_1}(C(X) \otimes P \otimes \bC^I) \to p_{\bv_2}(C(X) \otimes P \otimes \bC^I),\] 
which is invertible since so is $(p_{\bv_2}p_{\bv_1})^*p_{\bv_2}p_{\bv_1} =p_{\bv_1}p_{\bv_2}p_{\bv_1}$.
Then the unitary component $w:=p_{\bv_2} p_{\bv_1}(p_{\bv_1}p_{\bv_2}p_{\bv_1})^{-1/2}$ satisfies $w^*w=p_{\bv_1}$ and $ww^*=p_{\bv_2}$. 
We remark that $\| (p_{\bv_1}p_{\bv_2}p_{\bv_1})^{-1/2}-p_{\bv_1} \| \leq |I|^2\varepsilon$ holds. Indeed, we have $|(t^{-1/2}-1)'| = |-t^{-3/2}/2| \leq  1$ on $[2/3, 4/3]$, and hence $|t^{-1/2}-1| \leq |I|^2\varepsilon$ on $[1-|I|^2\varepsilon, 1+|I|^2\varepsilon]$. Therefore we obtain that
\begin{align} \| w(x) -p_{\bv_2}p_{\bv_1} \| = \| p_{\bv_2}p_{\bv_1}((p_{\bv_1}p_{\bv_2}p_{\bv_1})^{-1/2} - p_{\bv_1}) \| < |I|^2\varepsilon.  \label{form:polar} \end{align}

Now we define the family $\{ \bar{u}_\mu \}_{\mu \in I}$ as
\[ \bar{u}_\mu(x) := (\psi_{\mu}^{\bv_2}(x))^* w(x) \psi_{\mu}^{\bv_1}(x). \]
This $\{ \bar{u}_\mu\}$ is contained in $\cG_{C_1\varepsilon}(\bu)$ since 
\begin{align*} 
\bar{u}_\mu(x) v_{\mu\nu}^1(x) \bar{u}_\nu(x)^* =& \psi_{\mu}^{\bv_2}(x)^* w(x) \psi_{\mu}^{\bv_1}(x)\psi_{\mu}^{\bv_1}(x)^* \psi_{\nu}^{\bv_1}(x)  \psi_{\nu}^{\bv_1}(x)^* w(x)^* \psi_{\nu}^{\bv_2}(x) \\
 =&\psi_{\mu}^{\bv_2}(x)^* \psi_{\nu}^{\bv_2}(x)  = v_{\mu\nu}^2(x)
\end{align*}
and
\begin{align*}
\| \bar{u}_\mu(x) - 1\| & \leq \|(\psi_{\mu}^{\bv_2}(x))^* (w(x) -p_{\bv_2}p_{\bv_1}) \psi_{\mu}^{\bv_1}(x)  \| + \| (\psi_{\mu}^{\bv_2}(x))^*\psi_{\mu}^{\bv_1}(x) -1 \| \\
& < |I|^2 \varepsilon + \Big\| \sum_\mu \eta_{\mu}(x)^2 (v_{\nu\mu}^2(x)^*v_{\nu\mu}^1(x) -1) \Big\| \\
&< (|I|^2+1)\varepsilon = C_1\varepsilon.
\end{align*}

To see (2), let us fix $\bar{u}=\{ \bar{u}_\mu \} \in \cG_{C_1\varepsilon}(\bu)$. 
Let $B$ denote the C*-algebra 
\[ \Big\{ \{ h_\mu\}_{\mu \in I} \in \prod_{\mu \in I} C_b(U_\mu, \bB(P)) \mid v_{\nu\mu}(x)h_{\mu}(x)v_{\mu\nu}(x)=h_{\nu}(x) \ \forall x \in U_{\mu\nu} \Big\} \]
and let $B_{\mathrm{sa}, r} := \{ b \in B \mid b=b^*,  \| b\|<r \}$ for $r>0$. Set $\delta :=  4\sin ^{-1} (C_1\varepsilon /2)$. Then, $e(\{ h_\mu \}):= \{ \bar{u}_\mu e^{ih_\mu} \}_{\mu \in I}$ gives a continuous map $e \colon B_{\mathrm{sa}, \delta } \to \cG_{3C_1\varepsilon }(\bu)$. Moreover, since any $\bar{u}' \in \cG_{C_1 \varepsilon}(\bu)$ satisfies $\| \bar{u}_\mu - \bar{u}_\mu' \|<2C_1 \varepsilon$, we have $\bar{u}'=e(-i\log (\bar{u}_\mu^* \bar{u}_\mu'))$. 
That is, we obtain
\[ \cG_{C_1 \varepsilon }(\bu) \subset e(B_{\mathrm{sa}, \delta} ) \subset \cG_{3C_1 \varepsilon}(\bu). \]
Now we get the conclusion since $e(B_{\mathrm{sa}, \delta} )$ is contractible.
\end{proof}

For an open cover $\cU $ of $X$ and a closed subset $Y \subset X$, we write $\cU|_Y$ for the open cover $\{ U_\mu \cap Y \} _{\mu \in I_Y}$ of $Y$, where $I_Y:= \{ \mu \in I \mid U_\mu \cap Y \neq \emptyset\}$. For a \v{C}ech $1$-cocycle $\bv$ on $\cU$, we write $\bv|_Y$ for the restriction $\{ u_\mu |_{U_\mu \cap Y } \} _{\mu \in I_Y}$.
\begin{defn}
Let $(X,Y)$ be a pair of compact spaces with a finite open cover $\cU=\{ U_\mu \}_{\mu \in I}$.
An \emph{$(\varepsilon , \cU )$-flat stably relative bundle} on $(X,Y)$ with the typical fiber $(P , Q)$ is a quadruple $\fv:=(\bv_1,\bv_2,\bv_0, \bu)$, where 
\begin{itemize}
\item $\bv_1 $ and $\bv_2 $ are $(\varepsilon, \cU)$-flat $P$-bundle on $X$, 
\item $\bv_0$ is a $(\varepsilon , \cU _Y)$-flat $Q$-bundle on $Y$ and
\item $\bu \in \Hom_{\varepsilon} (\bv_1|_Y \oplus \bv_0 ,  \bv_2|_Y \oplus \bv_0)$.
\end{itemize}
We write the set of $(\varepsilon , \cU )$-flat stably relative bundles on $(X,Y)$ with the typical fiber $(P , Q)$ as $\Bdl _{P,Q}^{\varepsilon , \cU}(X,Y)$. 
\end{defn}
In the particular case that $Q=0$, we simply call a triple $\fv=(\bv_1,\bv_2 , \bu)$ an \emph{$(\varepsilon , \cU)$-flat relative bundle} and write as $\fv \in \Bdl_P^{\varepsilon, \cU}(X,Y)$.  
Our primary concern is a \emph{$(\varepsilon , \cU)$-flat stably relative vector bundle}, that is, a $(\varepsilon , \cU)$-flat stably relative bundle of Hilbert $\bC$-modules with the typical fiber $(\bC^n, \bC^m)$. 

\begin{defn}\label{defn:af}
For $0< \varepsilon < (3C_1)^{-1}$, we associate the $\K$-theory class
\[ [\fv]:=[E_{\bv_1}, E_{\bv_2}, E_{\bv_0}, \bar{u}] \in \K^0(X,Y;A) \]
to $\fv = (\bv_1 , \bv_2 , \bv_0 , \bu) \in \Bdl_{P,Q}^{\varepsilon , \cU}(X,Y)$, where $\bar{u}$ is an arbitrary element of   $\cG_{C_1\varepsilon} (\bu)$.
\end{defn}
This definition is independent of the choice of $\bar{u}$ by Lemma \ref{lem:gauge} (2).

\begin{rmk}\label{rmk:bdlequiv}
The associated K-theory class in Definition \ref{defn:af} depends only on unitary the equivalence class of $\fv$.
For $\bv \in \Bdl_P^{\varepsilon , \cU}(X)$ and $\bu \in \mathrm{U}(P)^I$, we say that 
\[ \bu \cdot \bv:= \{ u_\mu v_{\mu \nu} u_\nu ^* \}_{\mu,\nu \in I} \]
is unitary equivalent to $\bv$. Since $\bv$ and $\bu \cdot \bv$ are cohomologous as \v{C}ech $1$-cocycles, $E_{\bv}$ and $E_{\bu \cdot \bv}$ determine the same $\K$-theory class. 
Similarly, we say that $\fv \in \Bdl_{P,Q}^{\varepsilon , \cU}(X,Y)$ is unitary equivalent to $\fu \cdot \fv := (\bu_1 \cdot \bv_1 , \bu_2 \cdot \bv_2 , \bu _0 \cdot \bv_0 , \fu \cdot \bu)$ for $\fu:=(\bu_1,\bu_2,\bu_0) \in \mathrm{U}(P)^I \times \mathrm{U}(P)^I \times \mathrm{U}(Q)^I$, where  
\[ \fu \cdot \bu  := \{ \diag (u_{1,\mu} , u_{2, \mu }, u_{0, \mu })u_\mu \diag (u_{1,\mu} , u_{2, \mu }, u_{0, \mu })^* \}_{\mu \in I_Y}. \]
Then $\fu$ induces an isomorphism of the underlying stably relative bundles. In particular we have $[\fv] = [\fu \cdot \fv] \in \K^0(X,Y;A)$. 
\end{rmk}

Next, we define the (resp.\ stably) almost flat $\K_0$-group $\K^0_{\mathrm{af}} (X,Y; A)$ (resp.\ $\K^0_{\mathrm{s\mathchar`-af}}(X,Y;A)$) as subgroups of $\K^0(X,Y;A)$ and study their permanence property with respect to the pull-back. The discussion is inspired from the work by Hunger~\cite{mathGT160707820}. 

Let us fix a point $x_{\mu \nu} \in U_{\mu \nu}$ for each $\mu, \nu \in I$ with $U_{\mu\nu} \neq \emptyset$ in the way that $x_{\mu \nu }=x_{\nu\mu}$.
\begin{lem}\label{lem:ineq}
Let $\mu , \nu , \sigma \in I$ such that $U_{\mu\nu\sigma}:=U_\mu \cap U_\nu \cap U_\sigma \neq \emptyset$. Then, for $\bv \in \Bdl _P^{\varepsilon ,\cU}(X)$, we have 
\[\| v _{\mu\nu}(x_{\mu\nu}) v_{\nu\sigma}(x_{\nu\sigma} ) - v_{\mu\sigma} (x_{\mu\sigma}) \| < 3\varepsilon .\]
\end{lem}
\begin{proof}
Let us choose a point $x \in U_{\mu\nu \sigma}$. Then, 
\begin{align*}
&\| v _{\mu\nu}(x_{\mu\nu}) v_{\nu\sigma}(x_{\nu\sigma} ) - v_{\mu\sigma} (x_{\mu\sigma}) \|\\ < & \| v _{\mu\nu}(x_{\mu\nu}) v_{\nu\sigma}(x_{\nu\sigma} ) - v_{\mu\nu}(x)v_{\mu\sigma} (x)\| + \| v_{\mu\sigma} (x_{\mu\sigma})- v_{\mu\sigma} (x)  \| \\
	<& 2\varepsilon + \varepsilon = 3\varepsilon. \qedhere
\end{align*}
\end{proof}

\begin{lem}\label{lem:loctriv}
Let $X$ be a locally compact space with $\pi_1(X)=0$ and let $\cU$ be its finite good open cover. Then, there is a constant $C_2=C_2(\cU)$ depending only on $\cU$ such that $\Hom_{C_2  \varepsilon}( \mathbf{1}, \bv)$ is non-empty for any $\bv \in \Bdl_P^{\varepsilon, \cU}(X)$. 
\end{lem}
\begin{proof}
Let $N_\cU$ denote the nerve of $\cU$. For $\mu , \nu \in I$ with $U_{\mu\nu} \neq \emptyset$, we write $\langle \mu , \nu \rangle$ for the corresponding $1$-cell of $N_\cU$ whose direction is from $\nu$ to $\mu$. Let us fix a maximal subtree $T$ of $N_\cU$ and a reference point $\mu_0 \in I$. Then, for each $\mu \in I$ there is a unique minimal oriented path $\ell_\mu$ in $T$ from $\mu_0$ to $\mu$. Since $\cU$ is an good open cover, $X$ is homotopy equivalent to $N_\cU$ and in particular we have $\pi_1(N_\cU)=0$. Therefore, the closed loop $\ell_\mu^{-1} \circ \langle \mu, \nu \rangle \circ \ell_\nu$ is written of the form 
\begin{align} \prod_{i=1}^{C_{\mu \nu}} \ell_{\nu_i}^{-1} \circ \langle \mu_i , \sigma_i\rangle \circ  \langle \sigma_i , \nu_i \rangle  \circ \langle \nu_i , \mu_i \rangle  \circ \ell_{\nu_i}, \label{form:pathcomp}
\end{align}
where each $\mu_i , \nu_i, \sigma_i \in I$ satisfies $U_{\mu_i  \nu_i \sigma_i} \neq \emptyset $ (that is, $\{ \mu_i , \nu_i, \sigma_i \}$ is a $2$-cell of $N_\cU$).

For each $\mu \in I$, let $\mu_1,\dots, \mu_k \in I$ be the $0$-cells of $T$ such that $\ell_\mu := \langle \mu_{k} , \mu_{k-1} \rangle \circ \dots \circ \langle \mu_{1} , \mu_{0} \rangle $ and set
\begin{align}
u_\mu:=v_{\mu_{k}\mu_{k-1}}(x_{\mu_{k}\mu_{k-1}})v_{\mu_{k-1}\mu_{k-2}}(x_{\mu_{k-1}\mu_{k-2}}) \dots v_{\mu_{1}\mu_{0}}(x_{\mu_{1}\mu_0}) . \label{form:path}
\end{align}
By Lemma \ref{lem:ineq} and (\ref{form:pathcomp}), we get 
\[\| u_{\mu}v_{\mu \nu }u_\nu^* -1 \| <  3C_{\mu\nu}\varepsilon. \]
Now the proof is completed by choosing $C_2(\cU) := 3\max _{\mu , \nu \in I}C_{\mu\nu}$. 
\end{proof}

\begin{prp}\label{prp:subdiv}
Let $\cU = \{ U_\mu\}_{\mu \in I}$ be a finite good open cover of $X$. Assume that there is a subset $J \subset I$ such that $\cV:= \{U_\mu \}_{\mu \in J}$ also covers $X$. Then there is a constant $C_3=C_3(\cU, \cV)$ depending only on $\cU$ and $\cV$ such that the following hold. 
\begin{enumerate}
\item For any $\bv \in \Bdl^{\varepsilon , \cV}_{P}(X)$ there is $\tilde{\bv}=\{ \tilde{v}_{\mu \nu}\}_{\mu , \nu \in I} \in \Bdl^{C_3 \varepsilon, \cU}_P(X)$ such that $\tilde{v}_{\mu \nu}=v_{\mu \nu}$ for any $\mu , \nu \in J$. 
\item Let $\bv, \bv' \in \Bdl ^{\varepsilon , \cV}_P(X)$ with $\tilde{\bv}, \tilde{\bv}' \in \Bdl^{C_3\varepsilon, \cU}_P(X)$ constructed in (1). For $\bu \in \Hom_\varepsilon (\bv ,  \bv')$, there is $\tilde{\bu} \in \Hom_{(4C_3+1)\varepsilon}( \tilde{\bv}, \tilde{\bv}')$ such that $\tilde{u}_\mu =u_\mu$ for any $\mu \in J$.
\end{enumerate}
\end{prp}
\begin{proof}
For $\sigma \in I \setminus J$, let $\cU_\sigma$ be the open cover $\{ U_\sigma \cap U_\mu\}_{\mu \in I}$ of $U_\sigma$. Let $C_\sigma :=C_1(\cU_\sigma)C_2(\cU _\sigma )$, where $C_1(\cU_\sigma)$ and $C_2(\cU_\sigma)$ are the constants as in Lemma \ref{lem:gauge} and Lemma \ref{lem:loctriv} respectively. Let $C_3(\cU, \cV):= 2\max _{\sigma \in I \setminus J} C_\sigma $. 

First we show (1). For $\sigma \in I \setminus J$, we apply Lemma \ref{lem:loctriv} to the restriction $\bv|_{U_\sigma}=\{ v_{\mu\nu}^\sigma := v_{\mu \nu}|_{U_{\mu \sigma}} \}$ to get a morphism $\bu^{\sigma} \in \Hom _{C_2(\cU_\sigma)\varepsilon} ( \mathbf{1}, \bv|_{U_\sigma })$. Let $\bar{u} \in \cG_{C_1(\cU_\sigma)C_2(\cU_\sigma) \varepsilon } (\bu)$. Then, $\tilde{\bv}:=\{ \tilde{v}_{\mu \nu} \}_{\mu, \nu \in I}$ defined by
\[
\tilde{v}_{\mu \nu}(x):=
\left\{
\begin{array}{ll}
v_{\mu\nu}(x) & \text{ if $\mu, \nu \in J$, }\\
u_{\nu}^{\mu}(x) & \text{ if $\mu \in J$ and $\nu \not \in J$, }\\
u_{\nu}^{\sigma}(x)^*u_{\mu}^\sigma(x) & \text{ if $\mu, \nu \not \in J$}.
\end{array}
\right.
\]
is a desired \v{C}ech $1$-cocycle. 

Next we show (2). For each $\mu \in I \setminus J$, we fix $\sigma _{\mu} \in J$ such that $U_{\mu \sigma_\mu} \neq \emptyset $. Let
\[\tilde{u}_{\mu}:= \tilde{v}_{\mu \sigma_\mu}'(x_{\mu \sigma _\mu}) u_{\sigma_\mu} \tilde{v}_{\mu \sigma_\mu}(x_{\mu \sigma_\mu})^*. \]
Then,
\begin{align*}
&\|\tilde{u}_\mu \tilde{v}_{\mu \sigma_\mu }(x)\tilde{u}_{\sigma_{\mu}}^*  - \tilde{v}'_{\mu \sigma_{\mu} }(x) \| \\
<& \|\tilde{u}_\mu \tilde{v}_{\mu \sigma_\mu }(x)\tilde{u}_{\sigma_{\mu}}^*  - \tilde{u}_\mu \tilde{v}_{\mu \sigma_\mu }(x_{\mu \sigma _\mu})\tilde{u}_{\sigma_{\mu}}^*  \| + \| \tilde{v}_{\mu \sigma_{\mu} }'(x_{\mu \sigma _\mu})  - \tilde{v}'_{\mu \sigma_{\mu} }(x) \|\\
<& 2C_3 \varepsilon
\end{align*} 
and hence
\begin{align*}
&\| \tilde{u}_\mu \tilde{v}_{\mu \nu}(x)\tilde{u}_\nu^*  - \tilde{v}'_{\mu \nu }(x) \|\\
\leq & \|\tilde{u}_\mu \tilde{v}_{\mu \sigma_\mu }(x)\tilde{u}_{\sigma_{\mu}}^*  - \tilde{v}'_{\mu \sigma_{\mu} }(x) \| + \| \tilde{u}_{\sigma_\mu} \tilde{v}_{\sigma_\mu \sigma_\nu}(x)\tilde{u}_{\sigma_\nu}^*  - \tilde{v}'_{\sigma_\mu \sigma_\nu }(x) \| \\
&+ \| \tilde{u}_{\sigma_\nu} \tilde{v}_{\sigma_\nu \nu}(x)\tilde{u}_\nu^*  - \tilde{v}'_{\sigma_\nu \nu }(x) \| \\
<& (4C _3 + 1) \varepsilon. \qedhere
\end{align*}
\end{proof}

Let $(X,Y)$ be a pair of finite CW-complexes. In this paper we call $\cU$ a good open cover of the pair $(X,Y)$ if it is a good open cover of $X$ such that $\cU|_Y$ is also a good open cover of $Y$. Such an open cover exists because $(X, Y)$ is homotopy equivalent to a pair of finite simplicial complexes. For a pair of simplicial complexes, the family of open star neighborhoods of $0$-cells satisfies the desired property.

\begin{defn}\label{defn:af}
Let $(X,Y)$ be a pair of finite CW-complex and let $\cU$ be a finite good open cover of $(X,Y)$. An element $\xi \in \K^0(X,Y ; A)$ is \emph{(resp.\ stably) almost flat} with respect to $\cU$ if for any $\varepsilon >0$ there is a $(\varepsilon, \cU)$-flat (resp.\ stably) relative vector bundle $\fv$ of finitely generated projective Hilbert $A$-modules such that $x=[\fv]$. 
\end{defn}

\begin{cor}\label{cor:covering}
The subgroup consisting of all (resp.\ stably) almost flat elements of $\K^0(X,Y;A)$ is independent of the choice of good open covers.
\end{cor}
We write $\K_{\mathrm{af}}^0(X,Y;A)$ (resp.\ $\K_{\mathrm{s\mathchar`-af}}^0(X,Y;A)$) for the subgroup of (resp.\ stably) almost flat elements.
\begin{proof}
Let $\cU$ and $\cV$ be two open covers and $\cW := \cU \cup \cV$. Assume that $\xi \in \K^0(X,Y ; A)$ is represented by an $(\varepsilon, \cU)$-flat stably relative vector bundle $\fv=(\bv_1,\bv_2,\bv_0,\bu)$. By Proposition \ref{prp:subdiv} (1), we get $(C_3\varepsilon , \cW)$-flat bundles $\bw_1$, $\bw_2$ and $\bw_0$. Moreover, by Proposition \ref{prp:subdiv} (2), $\bu$ can be extended to $\tilde{\bu} \in \Hom_{(4C_3+1)\varepsilon}(\bw_1|_Y \oplus \bw_0,  \bw_2|_Y \oplus \bw_0)$. Finally, its restriction to $\cV$ is a $((4C_3+1)\varepsilon, \cV)$-flat stably relative bundle representing $\xi$.
\end{proof}

\begin{cor}\label{cor:pull}
Let $f$ be a continuous map from $(X_1,Y_1)$ to $(X_2,Y_2)$. If $\xi \in \K^0(X_2,Y_2 ; A )$ is almost flat, then so is $f^*\xi \in \K^0(X_1,Y_1 ; A)$. In particular, the subgroups $\K_{\mathrm{af}}^0(X,Y;A)$ and $\K_{\mathrm{s\mathchar`-af}}^0(X,Y;A)$ are homotopy invariant. 
\end{cor}
\begin{proof}
Let $\fv=(\bv_1, \bv_2,\bv_0,\bu) \in \Bdl_{P,Q}^{\varepsilon , \cU}(X,Y)$ be a $(\varepsilon, \cU)$-flat representative of $\xi$. Let us choose a good open cover $\cV=\{ V_{\nu} \}_{\nu \in J}$ of $(X,Y)$ which is a subdivision of $f^*\cU$. Let $\bar{f} \colon J \to I$ be a map with the property that $V_\nu \subset f^*U_{f(\nu)}$. Then, $f^*\fv:=(f^*\bv_1, f^*\bv_2, f^*\bv_0, f^*\bu)$ defined as $f^*\bv_i:=\{ f^*v_{\bar{f}(\mu), \bar{f}(\nu)} \}_{\mu , \nu \in J}$ for $i=0,1,2$ and $f^*\bu:= \{ u_{\bar{f}(\mu)} \}_{\mu \in J}$ is a $(\varepsilon , \cV)$-flat bundle on $(X_1, Y_1)$ representing $f^*\xi$.
By Corollary \ref{cor:covering}, $f^*\xi$ is almost flat with respect to an arbitrary good open cover of $(X_1, Y_1)$.
\end{proof}

Finally we define the infiniteness of (C*)-K-area for a relative K-homology cycle as a generalization of non-relative case introduced in \cite{MR1389019,MR3289846}, which is also independent of the choice of good open cover $\cU$ by Proposition \ref{prp:subdiv} in the same way as (the proof of) Corollary \ref{cor:covering}.
\begin{defn} 
Let $(X,Y)$ be a finite CW-complex and let $\xi \in \K_0(X,Y)$. 
\begin{enumerate}
\item We say that $\xi$ has \emph{infinite (resp.\ stably) relative $\K$-area} if there is an (resp.\ stably) almost flat $\K$-theory class $x \in \K^0(X,Y)$ such that the index pairing $ \langle x, \xi \rangle$ is non-zero.
\item Let $\cU$ be a good open cover of $(X,Y)$. We say that $\xi $ has \emph{infinite (resp.\ stably) relative C*-$\K$-area} if for any $\varepsilon>0$ there is a C*-algebra $A_\varepsilon$ and a (resp.\ stably) relative $(\varepsilon , \cU)$-flat bundle $\fv$ of finitely generated projective Hilbert $A_\varepsilon$-modules such that the index pairing $ \langle [\fv], \xi \rangle \in \K_0(A_\varepsilon)$ is non-zero. 
\end{enumerate}
A compact spin manifold $M$ with the boundary $N$ has (stably) relative (C*-)K-area if so is the K-homology fundamental class $[M,N] \in \K_*(M,N)$.
\end{defn}

\section{Comparing topological and smooth almost flatness}\label{section:4}
The notion of almost flat bundle is originally introduced in \cite{MR1042862} in terms of Riemannian geometry of connections in the following way.
Let $(M,g)$ be a compact Riemannian manifold with a possibly non-empty boundary. A pair $\be = (E, \nabla)$ is a \emph{smooth $(\varepsilon ,g)$-flat vector bundle} on $M$ if $E$ is a hermitian vector bundle on $M$ and $\nabla$ is a hermitian connection on $E$ whose curvature tensor $R^\nabla \in \Omega ^2(M,\End E)$ satisfies
\[ \| R^\nabla \|:= \sup _{x \in M} \sup _{\xi \in \bigwedge ^2 T_xM \setminus \{ 0 \} } \frac{\| R^\nabla(\xi )\|_{\End (E_x) } }{\| \xi \| }  <\varepsilon.\]
An element $x \in \K^0(M)$ is said to be \emph{almost flat} (in the smooth sense) if for any $\varepsilon >0$ there is a pair of smooth $(\varepsilon , g)$-flat vector bundles $\be_1=(E_1,\nabla_1)$ and $\be_2=(E_2,\nabla_2)$ such that $x=[E_1]-[E_2]$.
It is proved in \cite[Proposition 3]{MR3101798} that almost flatness of an element of the $K_0$-group is independent of the choice of the Riemannian metric $g$ on $M$. 

\begin{defn}
For two smooth $(\varepsilon ,g)$-flat vector bundles $\be_1$ and $\be_2$ on $(M, g)$, a morphism of smooth $(\varepsilon ,g)$-flat bundles from $\be_1$ to $\be_2$ is a unitary bundle isomorphism $u \colon E_1 \to E_2$ with
\[\| u \nabla _1 u^* - \nabla _2 \|_{\Omega^1} < \varepsilon, \]
where $\| \cdot \|_{\Omega^1}$ is the uniform norm on $\Omega ^1(M , \End (E_2))$.
\end{defn}

\begin{defn}
Let $(M,g)$ be a compact Riemannian manifold with the boundary $N$. For $n \geq 1$ and $m \geq 0$,
a \emph{smooth $(\varepsilon ,g)$-flat stably relative vector bundle} of rank $(n,m)$ on $(M,N)$ is a quadruple $\fe=(\be_1, \be_2, \be_0 , u)$, where
\begin{itemize}
\item $\be_1=(E_1 , \nabla_1)$ and $\be_2=(E_2,  \nabla_2)$ are rank $n$ smooth $(\varepsilon ,g)$-flat vector bundles on $M$, 
\item $\be_0= (E_0,\nabla_0)$ is a rank $m$ smooth $(\varepsilon ,g)$-flat vector bundle on $N$ and 
\item $u \colon \be_1|_N \oplus \be_0 \to \be_2|_N \oplus \be_0$ is a morphism of $(\varepsilon , g)$-flat bundles.
\end{itemize}
In the particular case of $m=0$, we simply call a triplet $\fe=(\be_1,\be_2,u)$ a smooth $(\varepsilon ,g)$-flat relative vector bundle of rank $n$. 
\end{defn}
We write $[\fe]$ for the element of $\K^0(X,Y)$ represented by the underlying stably relative vector bundle $(E_1,E_2,E_0,u)$.

\begin{lem}\label{lem:diff}
Let $(M,g)$ be a Riemannian manifold and let $x, y \in C^\infty (M, \bM_n)$ whose spectra (as elements of $C(M) \otimes \bM_n$) are included to a domain $D \subset \bC$. Let $\gamma$ be the boundary of a domain $D'supset \bar{D}$ and let $f$ be a holomorphic function defined on a neighborhood of $D'$. Then there is a constant $C_4=C_4(g,D,\gamma , f)$ depending only on $g$, $D$, $\gamma$ and $f$ such that
\[ \| d(f(x) - f(y))\|_{\Omega^1} \leq  C_4 (\| dx \|_{\Omega^1} \| x - y\| + \| dx-dy \| ), \]
where $\| \cdot \|_{\Omega^1}$ is the uniform norm on the space of matrix-valued $1$-forms $\Omega^1(M, \bM_n)$.
\end{lem}
\begin{proof}
The functional calculus $f(x)$ is given by the Dunford integral
\[ f(x) = \frac{1}{2\pi i}\int_{\lambda \in \gamma } f(\lambda) (\lambda -x)^{-1}d\lambda. \]
Since 
\begin{itemize}
\item $d((\lambda -x)^{-1})=-(\lambda -x)^{-1}(dx)(\lambda - x)^{-1} $ (which follows from $d((\lambda -x)(\lambda -x)^{-1}) =d(1)=0$),
\item $(\lambda -x)^{-1} - (\lambda-y)^{-1}= (\lambda-x)^{-1}(y-x)(\lambda -y)^{-1}$ and
\item $\| (\lambda - x)^{-1} \| \leq C_4':=\inf \{  d(\lambda , x) \mid \lambda \in \gamma , x \in D \}$ ,
\end{itemize}
we obtain that 
\begin{align*}
&\| d(f(x)-f(y))\| \\
 \leq& (2\pi)^{-1}\| f \|_{L^1} \sup_{\lambda \in \gamma} \| (\lambda -x)^{-1} dx (\lambda -x)^{-1} - (\lambda -y)^{-1} dy (\lambda -y)^{-1} \| \\
\leq & (2\pi)^{-1}\| f \|_{L^1} \Big( \sup_{\lambda \in \gamma} \| ((\lambda - x)^{-1} - (\lambda - y)^{-1}) dx (\lambda - x)^{-1}\|\\
&\hspace{5.1em} +   \sup_{\lambda \in \gamma} \| (\lambda - y)^{-1}dx((\lambda - x)^{-1} - (\lambda - y)^{-1}) \|\\
&\hspace{5.1em} +  \sup_{\lambda \in \gamma}  \| (\lambda -y)^{-1}(dx-dy)(\lambda -y)^{-1} \| \Big) \\ 
\leq & (2\pi)^{-1}\| f \|_{L^1} (2(C_4')^3 \| dx \| \| x-y \| + (C_4')^2 \| dx-dy \| ),
\end{align*}
where $\| f \|_{L^1}$ is the $L^1$-norm of $f$ on $\gamma$. Now the proof is completed by choosing $C_4$ as $(2\pi)^{-1}\| f \|_{L^1}(C_4')^2 \cdot \max \{ 2C_4' , 1 \}$.
\end{proof}

\begin{lem}\label{lem:polar}
Let $X$ be a finite CW-complex with an open cover $\cU$. For $0 < \varepsilon < 1/2$, let $\{ v'_{\mu \nu} \}_{\mu,\nu \in I}$ be a family of unitaries in $\bB(P)$ such that $\| v'_{\mu\nu} v'_{\nu \sigma} -v'_{\mu \sigma} \| < \varepsilon $. Let 
\begin{align*} \breve{\psi}_\mu(x) &:= \sum_{\nu \in I} \eta_\nu(x) \otimes v'_{\nu \mu } \otimes e_{\mu} \in C(X) \otimes \bM_n \otimes \bC^I , \\
v_{\mu \nu} (x) &:= (\breve{\psi}_{\mu}(x)^*\breve{\psi}_{\mu}(x))^{-1/2}\breve{\psi}_{\mu}(x)^*\breve{\psi}_{\nu}(x) (\breve{\psi}_{\nu}(x)^*\breve{\psi}_{\nu}(x))^{-1/2} ,
\end{align*}
where $\{ \eta _\mu \}_{\mu \in I}$ and $\{ e_\mu\}_{\mu \in I}$ be as in Remark \ref{rmk:bundle}. 
Then $\bv:=\{ v_{\mu \nu} \}_{\mu,\nu \in I}$ is a \v{C}ech $1$-cocycle satisfying $\| v_{\mu\nu}(x)-v'_{\mu\nu} \| < 4 \varepsilon$, and hence is $(8 \varepsilon , \cU)$-flat.
\end{lem}
\begin{proof}
Firstly,  
\[ \breve{\psi}_\mu(x)^*\breve{\psi}_\nu(x) - v'_{\mu\nu} = \sum_{\sigma \in I} \eta_{\sigma}(x)^2 (v'_{\mu\sigma} v'_{\sigma\nu}-v'_{\mu\nu} ) \]
implies $\| \breve{\psi}_\mu(x)^*\breve{\psi}_\nu(x) - v'_{\mu\nu} \| < \sum_\sigma \eta_{\sigma}^2(x) \| v'_{\mu\sigma} v'_{\sigma\nu}-v'_{\mu\nu} \| < \varepsilon $. In particular, we get $\| \psi_\mu (x)^*\psi_\mu(x) -1 \| < \varepsilon$, and hence 
\[ \| (\breve{\psi}_\mu (x)^*\breve{\psi}_\mu(x))^{-1/2} -1 \| < \|  \breve{\psi}_\mu (x)^*\breve{\psi}_\mu(x) -1 \| < \varepsilon \]
(here we use the fact $|z^{-1/2}-1| < |z -1|$ for $z \in [1/2,3/2]$). Therefore we get
\begin{align*}
 & \| v_{\mu\nu}(x) - v'_{\mu \nu} \| \\
  \leq  &\|  (\breve{\psi}_\mu (x)^*\breve{\psi}_\mu(x))^{-1/2} -1 \| \|  \breve{\psi}_\mu (x)^*\breve{\psi}_\nu(x) \| \| (\breve{\psi}_\nu (x)^*\breve{\psi}_\nu(x))^{-1/2} \| \\
  & + \|  \breve{\psi}_\mu (x)^*\breve{\psi}_\nu(x)  \| \|  (\breve{\psi}_\nu (x)^*\breve{\psi}_\nu(x))^{-1/2} -1 \| + \| \breve{\psi}_\mu (x)^*\breve{\psi}_\nu(x)  - v'_{\mu\nu} \| \\
  \leq & \varepsilon (1+\varepsilon)^2 + \varepsilon (1+\varepsilon) + \varepsilon < 4 \varepsilon.
 \end{align*}
 For the last inequality we use $\varepsilon <1/2$.
\end{proof}

\begin{lem}\label{lem:smoothing}
Let $0 < \varepsilon < 1/6$. Let $M$ be a Riemannian manifold with a finite open cover $\cU= \{ U_\mu \}_{\mu \in I}$. 
Then there exists a constant $C_5=C_5(g, \cU)$ depending only on $g$ and $\cU$ such that the following holds: For any $\bw \in \Bdl_{n}^{\varepsilon ,\cU}(M)$, there is $\bv \in \Bdl^{\cU, 24 \varepsilon}_{n}(M)$ such that 
\begin{itemize}
\item $ \| v_{\mu\nu}(x) -w_{\mu \nu}(x) \|<13 \varepsilon $ for any $x \in U_{\mu \nu}$ and 
\item each $v_{\mu \nu}$ is smooth and $\| d v_{\mu \nu} \|_{\Omega ^1(U_{\mu \nu} ,\bM_n)} <C_5 \varepsilon$.
\end{itemize}
\end{lem}
\begin{proof}
Let $\breve{\psi}_\mu$ and $v_{\mu \nu}$ be as in the statement of Lemma \ref{lem:polar} for $v'_{\mu\nu}=w_{\mu\nu}(x_{\mu\nu})$. 
By Lemma \ref{lem:ineq} and Lemma \ref{lem:polar} we obtain that $\{ v_{\mu\nu} \}_{\mu,\nu \in I} $ is $(24\varepsilon , \cU)$-flat and 
\[\| v_{\mu\nu}(x)-w_{\mu\nu}(x) \| \leq  \| v_{\mu \nu}(x) -w_{\mu\nu}(x_{\mu\nu}) \| + \|w _{\mu \nu}(x_{\mu\nu} )-w_{\mu\nu}(x) \| < 13\varepsilon .\]

Now we consider an estimate of the differential $dv_{\mu\nu}$. Let $\kappa := \max_{\mu}\| d\eta_\mu \|$. 
Note that $\| d(\eta_\mu^2) \| = \| 2\eta_\mu d \eta_\mu \| \leq 2\kappa$. Then we get
\begin{align*}
&\| d(\breve{\psi}_\mu^* \breve{\psi}_\nu) \| = \| d (\breve{\psi}_\mu^* \breve{\psi}_\nu - w_{\mu\nu} (x_{\mu\nu})) \| \\
\leq & \sum_{\sigma \in I} \| d(\eta_\sigma ^2)\| \cdot \| w_{\mu\sigma}(x_{\mu\sigma})w_{\sigma\nu}(x_{\sigma\nu})  - w_{\mu\nu}(x_{\mu\nu})\| < 2\kappa |I| \cdot 3 \varepsilon . 
\end{align*}

By the assumption $\varepsilon < 1/6$, we have that the spectrum $\sigma (\breve{\psi}_{\mu}(x)^*\breve{\psi}_{\mu}(x))^{-1/2}) $ is included to the interval $[1/2,3/2]$. Let $D$ and $D'$ be the open disk of radius $2/3$ and $3/4$ with the center $1$ respectively and let $\gamma = \partial D'$. We apply Lemma \ref{lem:diff} for $x=1$, $y=\breve{\psi}_{\mu}^*\breve{\psi}_{\mu}$, $D$ and $\gamma$ as above and $f(z)=z^{-1/2}$. Then we get a constant $C_4=C_4(g,D,\gamma , z^{-1/2})$ and an inequality $\| d((\breve{\psi}_{\mu}^*\breve{\psi}_{\mu})^{-1/2})\| < C_4 \cdot 6\kappa |I| \varepsilon $. Finally we obtain
\begin{align*}
&\| dv_{\mu\nu} \| \\
 \leq & \| d((\breve{\psi}_{\mu}^*\breve{\psi}_{\mu})^{-1/2}) \| \| \breve{\psi}_{\mu}^*\breve{\psi}_{\nu}\| \|(\breve{\psi}_{\nu}^*\breve{\psi}_{\nu})^{-1/2} \|  + \| (\breve{\psi}_{\mu}^*\breve{\psi}_{\mu})^{-1/2} \| \| d(\breve{\psi}_{\mu}^*\breve{\psi}_{\nu}) \| \| (\breve{\psi}_{\nu}^*\breve{\psi}_{\nu})^{-1/2} \|  \\
 &+ \| (\breve{\psi}_{\mu}^*\breve{\psi}_{\mu})^{-1/2} \| \| \breve{\psi}_{\mu}^*\breve{\psi}_{\mu} \| \| d((\breve{\psi}_{\nu}^*\breve{\psi}_{\nu})^{-1/2}) \|\\
<& 6C_4\kappa |I| \varepsilon \cdot (3/2) \cdot 2 + 6\kappa |I| \varepsilon \cdot 2 \cdot 2 + 6C_4\kappa |I|\varepsilon \cdot (3/2) \cdot 2  =  (36 C_4+ 24)\kappa |I|  \varepsilon .
\end{align*}
The proof is completed by choosing $C_5:= (36 C_4+ 24)\kappa |I| $. 
\end{proof}

\begin{lem}\label{lem:smoothingunitary}
Let $0<\varepsilon < \frac{1}{3C_1}$. There is a constant $C_6=C_6(\cU)$ depending only on $\cU$ such that the following holds: For $(\varepsilon , \cU)$-flat bundles $\bv_1$ and $\bv_2$ on $X$ with $\| dv_{\mu \nu}^i \| < \varepsilon$ (for $i=1,2$) and $\bu \in  \Hom_\varepsilon (\bv_1 , \bv_2) $, there is $\bar{u} \in \cG_{C_1\varepsilon}(\bu)$ such that $\| d\bar{u}_\mu \|_{\Omega^1} < C_6 \varepsilon$.
\end{lem}
\begin{proof}
Let $\psi^i_\mu:=\psi_\mu^{\bv_i}$ and $p_i:=p_{\bv_i}$ for $i=1,2$, $w$ and $\{ \bar{u}_\mu \}_{\mu \in I}$ be as in Remark \ref{rmk:bundle}. As in the proof of Lemma \ref{lem:gauge} (1), we may assume that $u_\mu =1$ for all $\mu \in I$. 
As in the proof of Lemma \ref{lem:smoothing}, let $\kappa := \max_\mu \| d\eta_\mu \|$. Then we have inequalities 
\begin{align*}
\| d\psi_\mu^i \| =& \| d\sum_\nu \eta_{\nu} v_{\mu \nu}^i \otimes e_\nu \| \\
\leq & \sum_\nu (\| d \eta _{\mu\nu}\|\|v_{\mu\nu}^i\| \|e_\nu \| + \|\eta_\nu \| \| dv_{\mu\nu}^i\|\|e_\nu\| )  \leq |I|(\kappa + (3C_1)^{-1}), \\
\| d((\psi ^{2}_\mu)^*\psi^{1}_\mu) \| =&\| d((\psi ^{2}_\mu - \psi_\mu^1)^*\psi^{1}_\mu) \| = \| d \sum_\nu \eta_\nu^2 (v_{\mu\nu }^2 - v_{\mu\nu}^1) ^*v_{\mu\nu}^1 \| \\
 \leq & \sum_\nu ( \|d(\eta_\nu^2)\| \| (v_{\mu\nu}^2 -v_{\mu\nu}^1 )^* \| \|v_{\mu\nu}^1\| + \| \eta_{\nu}^2 \| \| d(v_{\mu\nu}^2 -v_{\mu\nu}^1)^* \| \| v_{\mu\nu}^1 \| \\
 & \hspace{5em} + \| \eta_{\nu}^2\| \| (v_{\mu\nu}^2 -v_{\mu\nu}^1)^* \| \|  dv_{\mu\nu}^1\|) \\
\leq & |I| (2\kappa \varepsilon + 2\varepsilon +\varepsilon ^2)= (2\kappa +3) |I| \varepsilon , \\
\| dp_i \|  =& \| d\sum \eta _\mu \eta_\nu v_{\mu\nu}^i \otimes e_{\mu\nu} \|\\
\leq & \sum _{\mu,\nu} (\| d(\eta_\mu \eta _\nu) \| \| v_{\mu \nu }^i\|\|e_{\mu\nu} \| +  \| \eta_\mu \eta_\nu \| \| dv_{\mu\nu}^i\| \|e_{\mu\nu}\| ) \\
<& |I|^2(2\kappa  + (3C_1)^{-1}) ,\\
\| dp_1-dp_2 \| =& \| d\sum \eta_\mu \eta_\nu (v_{\mu\nu}^1 -v_{\mu\nu}^2) \otimes e_{\mu\nu}\| \\
\leq & \sum_{\mu,\nu} (\| d(\eta_\mu \eta_\nu) \| \| v_{\mu\nu}^1 -v_{\mu\nu}^2 \| \| e_{\mu\nu}\| +  \| \eta_\mu \eta_\nu\| \| dv_{\mu \nu}^1- dv_{\mu \nu}^2\|  \| e_{\mu\nu}\|) \\ 
< & |I|^2 (2\kappa \varepsilon +  2\varepsilon ) = |I|^2 (2\kappa +2) \varepsilon.
\end{align*}
Let $C_6'$ denotes the maximum of $ |I|(\kappa + (3C_1)^{-1})$ , $(2\kappa +3) |I|$ , $|I|^2(2\kappa  + \varepsilon)$ and $|I|^2(2\kappa + 2) $.

By the above inequalities together with (\ref{form:pdiff}), we get 
\begin{align*}
&\| d(p_1p_2p_1) - dp_1\| \\
=& \| d p_1\| \| p_2-p_1\| \| p_1\| + \| p_1 \| \| d(p_2-p_1) \| \| p_1\| + \| p_1\| \| p_2-p_1\| \| dp_1 \| \\
 <& C_6'  \cdot |I|^2\varepsilon + C_6' \varepsilon + C_6'  \cdot |I|^2\varepsilon = (2|I|^2+1)C_6' \varepsilon.
\end{align*}
We apply Lemma \ref{lem:diff} for $x=p_1$, $y=p_1p_2p_1$ (regarded as elements of $p_1 (C(X) \otimes \bB(P) \otimes \bM_I) p_1$) and $f(z)=z^{-1/2}$ as in Lemma \ref{lem:gauge} (1) and $D$, $\gamma $ as in Lemma \ref{lem:smoothing}. Then, together with (\ref{form:ppp}), we get a constant $C_4=C_4(g,D,\gamma ,z^{-1/2})$ and an inequality
\[ \| dp_1 - d(p_1p_2p_1)^{-1/2}  \| \leq C_4 ( C_6' \cdot |I|^2 \varepsilon + (2|I|^2+1)C_6' \varepsilon). \]
Therefore, we also get
\begin{align*}
&\|d( p_2p_1 (p_1p_2p_1)^{-1/2} - p_2 p_1) \| =\| d(p_2p_1((p_1p_2p_1)^{-1/2}-p_1)) \| \\
\leq &\| dp_2\| \| p_1\| \|(p_1p_2p_1)^{-1/2} -p_1\| + \| p_2 \| \| dp_1 \| \| (p_1p_2p_1)^{-1/2} -p_1 \| \\
& \hspace{3em} + \| p_2\| \|p_1\| \| d(p_1p_2p_1)^{-1/2} -dp_1 \| \\
\leq &  2\cdot  C_6'  \cdot |I|^2\varepsilon + C_4 ( C_6' \cdot |I|^2 \varepsilon + (2|I|^2+1)C_6' \varepsilon) =: C_6' \varepsilon .
\end{align*}
This inequality and (\ref{form:polar}) concludes the proof as
\begin{align*} 
&\| d\bar{u}_{\mu} \|  = \| d((\psi^2_\mu)^*w \psi_\mu^1)\| \\
\leq & \| d ((\psi^2_\mu)^* (w-p_2p_1) \psi_\mu^1) \| + \| d ((\psi^2_{\mu})^*\psi_\mu^1) \| \\
\leq & \| d(\psi_\mu^2)^* \| \|w-p_2p_1\|\| \psi_\mu^1\| +  \| (\psi_\mu^2)^*\| \| dw -d(p_2p_1) \| \| \psi_\mu^1\|  \\
& \hspace{3em} + \|(\psi_\mu^2)^*\| \| w-p_2p_1 \|\| d \psi_\mu^1 \|  + C_6' \varepsilon \\
\leq & C_6' \cdot |I|^2 \varepsilon + C_6'' \varepsilon + C_6' \cdot |I|^2\varepsilon  + C_6' \varepsilon =:C_6 \varepsilon. \qedhere
\end{align*}
\end{proof}

\begin{lem}\label{lem:smtop}
Let $(M,g)$ be a Riemannian manifold possibly with a collared boundary. Let $\cU:=\{ U_\mu \}_{\mu \in I}$ be an open cover of $M$ such that any two points $x,y$ in each $U_\mu$ is connected by a unique minimal geodesic in $U_\mu$. Then there is a constant $C_7=C_7(g, \cU)$ depending on $g$ and $\cU$ such that the following hold: 
\begin{enumerate}
\item Let $\be=(E,\nabla)$ be an $\varepsilon$-flat bundle on $M$. Then, there exists a local trivialization $\psi ^\be _\mu \colon U_\mu \times \bC^n \to E|_{U_\mu}$ such that the \v{C}ech $1$-cocycle 
\[ \bv^\be:=\{ v^\be_{\mu\nu}(x) := \psi^\be_\mu(x)^*\psi^\be_\nu(x) \}_{\mu,\nu \in I}\]
forms a $(C_7 \varepsilon , \cU)$-flat bundle.
\item Let $u \colon \be_1 \to \be_2$ be a morphism of $\varepsilon$-flat bundles. Then, 
 \[\bu:=\{ u_\mu := \psi_\mu^{\be_2}(x_\mu)^* u(x_\mu) \psi_\mu^{\be_1}(x_\mu)\}\]
forms a morphism of $(\varepsilon , \cU)$-flat bundles such that $u \in \cG_{C_7\varepsilon} (\bu)$.
\end{enumerate}
\end{lem}
For example, an open cover consisting of open balls of radius less than the injectivity radius of $M$ satisfies the assumption of Lemma \ref{lem:smtop} (when $M$ has a boundary, take an open cover of the invertible double $\hat{M}$ as above and restrict it to $M$). 
\begin{proof} 
Let $x,y \in U_\mu$. We write $[x,y]$ for the minimal geodesic connecting $x$ and $y$ in $U_{\mu}$ and 
\[ D_\mu (x,y):= \bigcup_{z \in [x,y]} [x_\mu,z]. \]
We define the constant $C_7$ as
\begin{align}
C_7:= \max_{\mu} \sup_{x,y \in U_\mu} \max \{ d(x,y), 2\mathrm{Area} (D_\mu(x,y))\} <\infty .\label{form:C7}
\end{align}

For a path $\ell \colon [0,t] \to M$, let $\Gamma _{\ell}^{\nabla} \colon E_{\ell(0)} \to E_{\ell(t)}$ denote the parallel transport along $\ell$. We fix an identification of $E_{x_\mu} $ with $\bC^n$. Then 
\[ \psi^\be_\mu (x):= \Gamma _{[x, x_\mu ]} \colon E_x \to   E_{x_\mu} \cong \bC^n \]
gives a local trivialization of $E$. Let $v_{\mu\nu}^\be(x):=\psi^\be_\nu(x)^*\psi^\be_{\mu}(x)$. Then $v_{\mu\nu}^\be(y)^*v_{\mu\nu}^\be (x)$ is the parallel transport along the boundary of the surface $D_\mu(x,y) \cup D_\nu (x,y)$. 
By a basic curvature estimate of the holonomy (see for example \cite[pp.19]{MR1389019}), we get
\[ \| v_{\mu\nu}^\be(y)^*v_{\mu\nu}^\be (x) -1\|  < \mathrm{Area}(D_\mu (x,y) \cup D_\nu(x,y)) \cdot \| R^\nabla \| < C_7 \varepsilon. \]

To see (2), it suffices to show that $\| \psi_{\mu}^{\be_2}(x)^*u(x)\psi_\mu^{\be_1}(x) - u_\mu \| <C_7 \varepsilon$. Let $x(t)$ denote the point of $[x,x_\mu]$ uniquely determined by $d(x,x(t))=t$. Since
\[ u\Gamma ^{ \nabla _1}_{[x_\mu,x]}u^* - \Gamma_{[x_\mu,x]}^{\nabla_2}=  \Gamma ^{u \nabla_1 u^*}_{[x,y]} - \Gamma_{[x_\mu,x]}^{\nabla_2} = \int_{0}^{d(x,x_\mu)}(u\nabla^1_{\frac{d}{dt}} u^* -\nabla^2_{\frac{d}{dt}})\Gamma_{[x_\mu,x(t)]}dt,  \]
we obtain that
\[\|\Gamma ^{\nabla_2}_{[x_\mu,x]} u \Gamma^{\nabla_1}_{[x, x_\mu]} -u \| = \|u\Gamma ^{ \nabla _1}_{[x_\mu,x]}u^* - \Gamma_{[x_\mu,x]}^{\nabla _2} \| \leq d(x,y)\varepsilon \leq C_7 \varepsilon . \qedhere \] 
\end{proof}

\begin{lem}\label{lem:topsm}
Let $(M,g)$ and $\cU$ be as in Lemma \ref{lem:smtop}. Then there is a constant $C_8=C_8(g, \cU)$ depending only on $g$ and $\cU$ such that the following hold for any $0<\varepsilon < \frac{1}{4C_8}$.
\begin{enumerate}
\item Let $\bv$ be a $(\varepsilon , \cU)$-flat vector bundle. Then, the underlying vector bundle $E_\bv$ admits an $(C_8 \varepsilon ,g)$-flat connection $\nabla_\bv$.
\item For $\bu \in \Hom_\varepsilon ( \bv_1, \bv_2)$, there is $\bar{u} \in \cG_{C_8 \varepsilon} (\bu)$ such that $\|\bar{u} \nabla_{\bv_1}\bar{u}^*-\nabla_{\bv_2} \|_{\Omega^1}<C_8\varepsilon$.
\end{enumerate}
\end{lem}
\begin{proof}
By Lemma \ref{lem:smoothing}, we may assume that $\{ v_{\mu\nu}\} $ is $(24 \varepsilon ,\cU)$-flat and $\| dv_{\mu\nu} \| <C_5 \varepsilon$. As in previous lemmas, let $\kappa :=  \max_{ \mu \in I} \| d\eta _\mu \|$. 

The connection
\[ \nabla _\bv^\mu  = d+a^\bv_\mu := \sum _\nu \eta_\nu^2 \cdot v_{\mu \nu} \circ d \circ v_{\mu \nu}^* =d+ \sum_\nu \eta_\nu^2 v_{\mu\nu}dv_{\mu\nu}^* \]
on the trivial bundle $\bC^n_{U_\mu}$ satisfies $v_{\mu\nu}^*\nabla_\bv^\mu v_{\mu\nu}= \nabla_\bv^\nu$ and hence gives rise to a connection $\nabla_\bv$ on $E$. 
Since $\| dv_{\mu\nu} \|<C_5 \varepsilon$, we have $\| a^\bv_\mu \wedge a^\bv_\mu \|\leq \| a_\mu^\bv \|^2 < (|I|C_5\varepsilon) ^2$ and
\[\| da^\bv_\mu \Big\|  \leq \|\sum_\nu d\eta_\nu \wedge v_{\mu\nu}dv^*_{\mu\nu} \Big\| + \Big\| \sum_\nu \eta _\nu dv_{\mu\nu} \wedge dv_{\mu\nu}^* \Big\| \leq  \kappa |I| C_5 \varepsilon + |I| (C_5 \varepsilon) ^2. \]
Therefore, $\| R_\nabla \| = \max_{\mu \in I}  \| da^\bv_\mu + a^\bv_\mu\wedge a^\bv_\mu \| \leq (|I|^2C_5^2 + \kappa |I|C_5 + |I|C_5^2 )\varepsilon$.

Next we show (2). Firstly, in the same way as the above paragraph  we replace $\bv_1$ and $\bv_2$ to $\bv_1'$ and $\bv_2'$ with $\| d v_{\mu \nu} \| <C_5 \varepsilon$ and $d(\bv_i,\bv_i')<13\varepsilon$ if necessary. Then we may assume $\bv_1$, $\bv_2$ satisfies $\| dv_{\mu  \nu} \|<C_5\varepsilon$ and $\bu \in \Hom _{27 \varepsilon} (\bv_1, \bv_2)$. Set $C_5':= \max \{ C_5 , 27 \}$. 
By Lemma \ref{lem:smoothingunitary}, there is $\bar{u} \in \cG_{C_1C_5'\varepsilon}(\bu)$ such that $\| d\bar{u}_\mu \| < C_6C_5'\varepsilon$. Then
\begin{align*}
\bar{u}_\mu \nabla_{\bv_1} \bar{u}^*_\mu &= \sum \eta_{\nu} \bar{u}_\mu v_{\mu\nu}^1 \circ d \circ v_{\nu\mu}^1 \bar{u}_\mu^* = \sum \eta_{\nu}v_{\mu\nu}^2 \bar{u}_\nu \circ d \circ \bar{u}_\nu^* v_{\nu\mu}^2\\
&=\nabla_{\bv_2} + \sum \eta_\mu v_{\mu\nu}^2 \bar{u}_\nu (d \bar{u}_\nu^*) v_{\nu\mu}^2
\end{align*}
implies $\|\bar{u} \nabla_{\bv_1}\bar{u}^*-\nabla_{\bv_2} \|_{\Omega^1}<|I| C_5'C_6 \varepsilon$. Now the proof is completed by  choosing $C_8:=\max \{ |I|^2(C_5')^2 + \kappa |I|C_5' + |I|(C_5')^2,  C_1C_5', |I|C_5'C_6 \}$.
\end{proof}

\begin{thm}\label{thm:equiv}
Let $M$ be a compact Riemannian manifold with the boundary $N$. An element $x\in \K^0(M,N)$ is almost flat in smooth sense if and only if it is almost flat in topological sense (i.e.,\ in the sense of Definition \ref{defn:af}). 
\end{thm}
\begin{proof}
By Lemma \ref{lem:smtop} and Lemma \ref{lem:topsm}, we can associate from smooth or topological $\varepsilon$-flat stably relative bundles to the other. Since this correspondence preserve the underlying stably relative bundle, we get the conclusion.
\end{proof}

\section{Enlargeability and almost flat bundle}\label{section:4.5}
A connected Riemannian manifold $(M, g)$ is said to be (resp.\ area-) enlargeable if for any $\varepsilon >0$ there is a connected covering $\bar{M}$ and an (resp.\ area-) $\varepsilon$-contracting map $f_\varepsilon$ with non-zero degree from $\bar{M}$ to the sphere $S^n$ with the standard metric, which is constant outside compact subset of $M$. Here we say that $f_\varepsilon$ is area-$\varepsilon$-contracting if $\|\Lambda^2 T_xf_\varepsilon \| \leq \varepsilon$ for any $x \in M_\infty$. Note that any enlargeable manifold is area-enlargeable.

\begin{thm}\label{thm:enlarge}
Let $(M,g)$ be a compact Riemannian spin manifold with a collared boundary $N$. If $M_\infty$ is area-enlargeable, then $M$ has infinite stably relative C*-K-area.
\end{thm}

Firstly we prepare some notations. For $M, N$ as above, let $M_r$ denote the space $M \sqcup_{N} N \times [0,r]$ and $N_r:= \partial M_r$ for $r \in [0,\infty]$. We choose an open cover of $M$ using $g$ as in Lemma \ref{lem:smtop}.    
Let $q_r$ denote the continuous map $M_r \to M$ determined by $q_r|_M=\id_M$ and $q_r|_{N \times [0,r]}$ is the projection to $N$. 
We define the open cover $\cU_k$ of $M_k$ as
\[ \cU_k:= \{ U_{(\mu,k)}:=q_r^*U_\mu \cap V_l \}_{(\mu, l) \in I \times k}, \] 
where $V_0=M_1^\circ$, $V_l= N \times (l-1,l+1)$ for $l=1,\dots , n-1$ and $V_k=N \times (k-1,k]$.  
Next, for a covering $\bar{\pi} \colon \bar{M} \to M$, we write $\bar{\cU}$ for the open cover of $\bar{M}$ consisting of connected components of $\pi^{-1}(U_\mu)$'s and $\bar{I}$ for the index set of $\bar{\cU}$. 
We use the same letter $\bar{\pi}$ for the canonical map $\bar{I} \to I$. Similarly we define $\bar{\cU}_k$ and $\bar{I}_k$.

\begin{lem}\label{lem:cylinder}
Let $k \in \bN$ and let $(\bv, \bw , \bu)$ be a $(\varepsilon, \cU_k)$-flat relative bundle with the typical fiber $P$ on $(M_k, N_k )$. 
Then there is a stably relative $(2\varepsilon , \cU)$-flat bundle $\fv$ of Hilbert $A$-modules on $(M,N)$ such that $[\fv] = [\bv, \bw , \bu ]$ under the canonical identification $\K^0(M,N;A) \cong \K^0(M_k , N_k ; A)$. 
\end{lem}
\begin{proof}
For $l=0,\dots, k$, we define a $(\varepsilon, \cU|_N)$-flat $P$-bundle $\bv_l$ on $N$ by
\[ \bv_l := \{v_{(\mu,l)(\nu,l)}|_{U_{(\mu,l)(\nu,l)} \cap N \times \{ l \}}  \}_{\mu , \nu \in I} \]
under the canonical identification of $(N, \cU|_N)$ with $(N \times \{ l \} , \cU_k |_{N \times \{ l \}} )$. Similarly we define $\bw_l$ for $l=0,\dots ,k$. 

For $l=0,\dots, k$, fix $x_{\mu,l} \in U_{\mu,l} \cap N \times \{ l + \frac{1}{2} \}$. We define $\bu_l=\{ u_{l,\mu} \}_{\mu \in I}$ by
\[ u_{l, \mu } := \left\{ \begin{array}{ll}v_{(\mu,l+1)(\mu , l)}(x_{\mu , l}) & l=0,\dots , k-1 \\
u_{(\mu , k)} & l=k, \\ w_{(\mu,2k-l)(\mu , 2k-l+1)}(x_{\mu , 2k-l+1}) & l=k+1,\dots , 2k. \end{array} \right. \] 
Then we have $\bu_{l} \in \Hom_{2\varepsilon }(\bv_{l} , \bv_{l+1} )$, $\bu_k \in \Hom_\varepsilon ( \bv_k , \bw_k)$ and $\bu_{2k-l} \in \Hom _{2\varepsilon}( \bw_{l+1} , \bw_{l})$ for $l=0,\dots ,k-1$.

Let $\tilde{\bv}_1$ and $\tilde{\bv}_2$ be restrictions of $\bv$ and $\bw$ to $M$ with the open cover $\cU_k|_M = \cU$ respectively. 
Let $Q=P^{2k}$, let $\tilde{\bv}_0:= \bv_1 \oplus \dots \oplus \bv_k \oplus \bw_k \oplus \dots \oplus \bw_1 $ and let $\tilde{\bu}= \{ \tilde{u}_\mu \}_{\mu \in I} $, where each $\tilde{u}_\mu \colon P \oplus Q \to P \oplus Q$ is determined by
\[ \tilde{u}_\mu (\xi_0, (\xi_1 , \dots , \xi_{2n})) = (u_{2n, \mu}\xi_{2n} , (u_{0 , \mu}\xi_0 , u_{1,\mu}\xi_1, \dots ,u_{2k-1, \mu}\xi_{2k-1}) ) \]
for $\xi_0 , \dots , \xi_{2k}  \in P$. Then we have
\[ \| \tilde{u}_\mu (\tilde{v}^1_{\mu \nu} \oplus \tilde{v}^0_{\mu \nu }) \tilde{u}_\nu^* - \tilde{v}^2_{\mu \nu}\oplus \tilde{v}^0_{\mu \nu} \| < 2\varepsilon,  \]
that is, $\fv:=(\tilde{\bv}_1,\tilde{\bv}_2,\tilde{\bv}_0,\bu)$ is a stably relative $(2\varepsilon, \cU)$-flat bundle with the typical fiber $(P,Q)$ on $(M,N)$. 

Finally we observe that $[\bv, \bw, \bu] = [\fv]$ in $\K^0(M , N ;A)$. Let $q_l \colon M \to M_l$ be a diffeomorphism extending the canonical identification $N \to \partial M_l$ and let $E_l:= q_l^*E_{\bv_l}$, $E_{2k-l+1}:=q_l^*E_{\bw_l}$ for $l=0,\dots ,k$. Note that $E_l|_N \cong E_{\bv_l}$. Let us choose $\{ \bar{u}_{l,\mu}\}_{\mu \in I} \in \cG_{2C_1 \varepsilon }(\bu_l)$ by Lemma \ref{lem:gauge}, which induces a unitary bundle isomorphism $\bar{u}_l \colon E_l|_N \to E_{l+1}|_N$. 
Then the K-theory class $[\fv]$ is represented by $[E_{\bv_1} , E_{\bv_2}, E_{1}|_N \oplus \dots \oplus  E_{2k}|_N , \bar{U}] $, where
\[ \bar{U} (\xi_ 0 , (\xi_1 , \dots ,\xi_{2k})) = (\bar{u}_{2k}\xi_{2k}, (\bar{u}_{0}\xi_0, \dots , \bar{u}_{2k-1}\xi_{2k-1})). \] 
Let $E_0:= E_{\bv_1}$ and $E_{2k+1}:= E_{\bv_{2}}$. Now we use the equivalence relations on $\Bdl_s(X,Y;A)$ discussed in pp.\pageref{item:equiv} to obtain
\begin{align*}
 [\fv]=&[E_{\bv_1} , E_{\bv_2}, E_{1}|_N \oplus \dots \oplus E_{2k}|_N , \bar{U}] \\
 =& [E_{\bv_1} \oplus E_1 \oplus \dots \oplus E_{2k} ,  E_1 \oplus \dots \oplus E_{2k} \oplus E_{\bv_2} , 0 , \bar{U}  ]  \\
 = & \sum_{l=0}^{2k} [E_l , E_{l+1}, 0, \bar{u}_l] = [E_{k} , E_{k+1} , 0 , \bar{u}_k  ] = q_k^*[\bv, \bw , \bu ]. \qedhere 
 \end{align*}
\end{proof}

Let $F \to \bar{M} \to M$ be a (possibly infinite) connected covering and extend it to $\bar{M}_\infty \to M_\infty$. Let $\sigma$ denote the monodromy representation of $\Gamma:= \pi_1(M)$ on $\ell^2(F)$ and let 
\[ A:=\{ (T,S) \in \bB(\ell^2(F)) ^{\oplus 2} \mid S \in \sigma(C^*(\Gamma)), \ T-S \in \bK \}. \]
Then the the exact sequence 
\begin{align} 0 \to \bK(\ell^2(F)) \xrightarrow{i} A \xrightarrow{\mathrm{pr}_2} \sigma (C^*(\Gamma)) \to 0,\label{form:Asplit}\end{align}
where $i$ is the embedding to the first component and $\mathrm{pr}_2$ is the projection to the second component, splits. 

For a complete Riemannian manifold $M$ with an open cover $\cU$ such that each $U_\mu$ is relatively compact, a \v{C}ech $1$-cocycle $\bv$ on $\cU$ is compactly supported if $v_{\mu \nu} \equiv 1$ except for finitely many $(\mu ,\nu) \in I^2 $ with $U_\mu \cap U_\nu \neq \emptyset$. If a \v{C}ech $1$-cocycle $\bv$ is supported in an open submanifold $M_0$, i.e., $v_{\mu \nu} \equiv 1$ for any $(\mu, \nu ) \in I^2$ with $U_\mu \cap U_\nu \not \subset M_0 \neq \emptyset$, we associate a relative bundle $(\bv |_{M_0}, \mathbf{1} , \mathbf{1})$ on $M_0$ with the open cover $\{ U_\mu \cap M_0\}$. 

\begin{lem}
Let $M_r$, $\bar{M}_r$ and $A$ be as above. Then there is a Hilbert $A$-module bundle $\cP$ on $M$ and a $\ast$-homomorphism $\theta  \colon C_0(\bar{M}_\infty) \to \bK(C(M_\infty, \cP))$ such that, for any compactly supported $(\varepsilon , \bar{\cU}_\infty)$-flat vector bundle $\bv  \in \Bdl_{n}^{\varepsilon , \bar{\cU}_\infty}(\bar{M}_\infty)$ (with the support included to $\bar{M}_r$), the corresponding element $\theta_*[\bv , \mathbf{1} , \mathbf{1}] \in \K^0(M_r , N_r ; A)$ is represented by an $(\varepsilon , \cU)$-flat bundle of finitely generated projective Hilbert $A$-modules. 
\end{lem}
\begin{proof}
Let $\hat{\sigma} \colon \Gamma \to \mathrm{U}(A)$ be the representation given by $\hat{\sigma}(\gamma ):= (\sigma (\gamma) ,\sigma (\gamma ))$. 
Let $\cA$ denote the C*-algebra bundle $\tilde{M}_r \times_{\Ad \hat {\sigma}} A$, which acts on the Hilbert bundle $\cH:=\tilde{M}_r \times _{\hat \sigma } (\ell^2(F)^{\oplus 2})$. 
Then $C(M_r,\cA)$ is isomorphic to $\bK(C(M_r,\cP))$, where $\cP:= \tilde{M}_r \times_{\hat{\sigma}} A$. 
Let $p_\cP := \sum \eta_\mu \eta_\nu \hat{\sigma}(\gamma _{\mu \nu}) \otimes e_{\mu \nu } \in C(M_r, A) \otimes \bM_I$ as in Remark \ref{rmk:bundle} and let $\tau$ denote the identification of $C(M_r,\cA)$ with the corner subalgebra $p_\cP (C(M_r, A) \otimes \bM_I)p_\cP$. 

The Hilbert space $L^2(M_r, \cH)$ is canonically isomorphic to $L^2(\bar{M}_r)^{\oplus 2}$. 
Moreover, the $\Gamma$-equivariant inclusion $c_0(F) \subset \bK(\ell^2(F)) \subset A$ induces 
\[ \theta \colon C_0(\bar{M}_r) \cong C(M_r, \cC) \to C(M_r, \cA), \]
where $\cC:=\tilde{M}_r \times _{\Ad \hat \sigma}c_0(F)$. We remark that it is extended to $\theta \colon C(M_r, \cC^+) \to C(M_r, \cA)$, where $\cC^+:=\tilde{M}_r \times_{\Ad \hat \sigma} c_0(F)^+$. Similarly we define $\theta_{\mu \nu} \colon C_b(U_{\mu \nu}, \cC^+) \to C_b(U_{\mu\nu} , \cA)$.

We fix a local trivialization
\[ \chi_\mu \colon L^2(U_\mu , \ell^2(F)^{\oplus 2}) \to L^2(U_\mu , \cH) \cong L^2(\bar{\pi}^{-1} (U_\mu))^{\oplus 2} \]
coming from that of the covering space $\varphi_\mu \colon U_\mu \times F \to \bar{\pi}^{-1}(U)$ as a fiber bundle with the structure group $\sigma (\Gamma)$. Then there is $\gamma _{\mu \nu} \in \Gamma$ for each $\mu , \nu \in I$ such that $\chi_\mu ^* \chi _\nu = \hat \sigma (\gamma_{\mu \nu})$. Then the $\ast$-homomorphism $\tau$ is written explicitly as
\[ \tau (f):= \sum_{\mu , \nu} \eta_\mu \eta_\nu \cdot \chi_\mu^* (f|_{U_{\mu\nu}}) \chi_\nu \otimes e_{\mu \nu}. \]

For an $(\varepsilon, \cU_\infty)$-flat bundle $\bv \in \Bdl_{n}^{\varepsilon , \bar{\cU}_\infty}(\bar{M} _\infty )$ supported in $\bar{M}_r$, let
\begin{align*}
\tilde{v}'_{\mu \nu} &:= \prod _{\substack{\bar{\pi}(\bar{\mu})=\mu , \bar{\pi}(\bar{\nu})=\nu \\ U_{\bar{\mu} \bar{\nu}}\neq \emptyset} } \diag (v_{\bar{\mu} \bar{\nu}} , 1) \in (C_b(U_{\mu \nu} , \cC ^+) \otimes \bM_n)^{\oplus 2},\\
 \tilde{v}_{\mu \nu}&:= \chi_\mu^* \theta_{\mu\nu} (\tilde{v}'_{\mu\nu})  \chi_\nu  \in C_b(U_{\mu \nu} , A) \otimes \bM_n
 \end{align*}
 for any $\mu , \nu \in I_r$. Then $\tilde{\bv} := \{ \tilde{v}_{\mu \nu} \}_{\mu,\nu \in I_r}$ is a \v{C}ech $1$-cocycle on $M_r$ taking value in the unitary group of $ A \otimes \bM_n$.
 Moreover, by the construction, $(\varepsilon, \bar{\cU})$-flatness of $\bv$ implies that $\tilde{\bv}:=\{ \tilde{v}_{\mu \nu} \}_{\mu, \nu \in I}$ is also an $(\varepsilon , \cU)$-flat bundle of Hilbert $A$-modules.

As in Remark \ref{rmk:bundle}, let 
\begin{align*}
p_{\tilde{\bv}}&:= \sum_{\mu,\nu} \eta_\mu \eta_\nu \otimes \tilde{v}_{\mu \nu} \otimes e_{\mu \nu} \in C(M, A) \otimes \bM_n \otimes \bM_I, \\
p_{\bv} &:=\sum_{\mu, \nu} \eta_{\mu} \eta_{\nu} \otimes \tilde{v}_{\mu \nu}' \otimes e_{\mu \nu}\in C(M_r, \cC^+) \otimes \bM_n \otimes \bM_{I_r}, \\
p_{\mathbf{1}}&:=\sum_{\mu, \nu} \eta_{\mu} \eta_{\nu} \otimes 1_n \otimes e_{\mu \nu}\in C(M_r, \cC^+) \otimes \bM_n \otimes \bM_{I_r},
\end{align*}
Then we have $[p_{\mathbf{1}}]=[1_n]$, $p_\bv - p_{\mathbf{1}}  \in C_0(M_r^\circ, \cC) \cong C_0(\bar{M}_r^\circ)$ and the difference element $[p_{\bv} , p_{\mathbf{1}}] \in \K_0(C_0(\bar{M}_r))$ is equal to $[\bv] - [1_n]$. Therefore, the remaining task is to show that $\theta_*([p_\bv] - [1_n])=[p_{\tilde{\bv}}] - [1_n]$.

The projection 
\begin{align*}
 (\tau \circ \theta)(p_\bv) &= \sum_{\sigma, \tau} \sum_{\mu , \nu} \eta_\sigma \eta_\tau \eta_\mu \eta_\nu \otimes \chi_{\sigma}^* \theta_{\mu\nu}(\tilde{v}_{\mu\nu}') \chi_\tau \otimes e_{\mu\nu} \otimes e_{\sigma \tau} \\
 &  \in C(M_r , A) \otimes \bM_n \otimes \bM_I \otimes \bM_I  
 \end{align*}
is equal to the projection as in Remark \ref{rmk:bundle} associated to the \v{C}ech $1$-cocycle $\{ \chi_\sigma \tilde{v}_{\mu \nu} \chi_\tau \}_{(\mu, \sigma), (\nu, \tau) \in I^2}$ on the open cover $\cU^2:= \{ U_{\mu \sigma} \}_{(\mu ,\sigma) \in I^2}$ and the square root of partition of unity $\{ \eta_{\mu} \eta_{\sigma}\}_{(\mu, \sigma) \in I^2}$. 
At the same time, if we use the square root of partition of unity $\{ \eta_{\mu} \delta_{\mu \sigma } \}$ (where $\delta_{\mu \sigma}$ denotes the delta function) instead of $\{ \eta_\mu \eta_\sigma\}$, then the corresponding projection is identified with $p_{\tilde{\bv}}$. That is, the support of $(\tau \circ \theta)(p_{\bv})$ is isomorphic to that of $p_{\tilde{\bv}}$. This concludes the proof. 
\end{proof}

\begin{proof}[Proof of Theorem \ref{thm:enlarge}]
By taking the direct product with $\bT^1$ if necessary, we may assume that $n:= \dim M$ is even. Let $E$ be a vector bundle on $S^n$ such that $c_n(E)=1$ and let us fix a hermitian connection. For $\varepsilon >0$, let $f_\varepsilon \colon \bar{M}_\infty \to S^n$ be an area-$\varepsilon$-contracting map with non-zero degree. Then the induced connection on $f_\varepsilon^*E$ with the pull-back connection is $(C\varepsilon , g)$-flat in the smooth sense, where the constant $C>0$ is the norm of the curvature of $E$. Let $k \in \bN$ such that $f_\varepsilon$ maps $N \times [k,\infty)$ to the base point $\ast$ of $S^n$.

By Lemma \ref{lem:smtop}, there is a local trivialization $\{ \psi_{\bar{\mu}} \}_{\bar{\mu} \in \bar{I}}$ of $f_\varepsilon^*E$ such that $\bv:=\{ v_{\bar{\mu}\bar{\nu}}=\psi_{\bar{\mu}}^* \psi_{\bar{\nu}} \}_{\mu,\nu \in I}$ is $(C_7\varepsilon , \bar{\cU}_r)$-flat.  Here we remark that the proof of Lemma \ref{lem:smtop} also works for the noncompact manifold $\bar{M}$ since the constant $C_7=C_7(g,\bar{\cU}_k)$ given in (\ref{form:C7}) actually coincides with $C_7(g,\cU_k)$. Note that we also have $C_7(g,\bar{\cU}_k)=C_7(g,\bar{\cU}_1)$, that is, there is a uniform upper bound for $C_7(g, \bar{\cU}_k)$'s.

The remaining task is to show that the pairing $\langle \theta_*[\bv , \mathbf{1} , \mathbf{1}] , [M,N] \rangle $ is non-trivial. 
For an even-dimensional connected manifold $X$, we write $\beta _X$ for the image of the Bott generator in $\K^0(X)$ by an open embedding. Then $[E] - [\bC^n] =\beta_{S^n} \in \K^0(S^n)$ and hence
\[ [\bv ,\mathbf{1} , \mathbf{1}] = f_\varepsilon ^* [E] - [\bC^n] = \mathrm{deg} f_\varepsilon \cdot \beta_{\bar{M}_\infty } \in \K^0(\bar{M}_\infty).  \]
Let us choose an open subspace $U$ of $M$ such that $\bar{\pi} ^{-1}(U) \cong U \times F$ and a copy $\bar{U} \subset \bar{\pi}^{-1}(U)$ of $U$. Then we have $C_0(U,\cA) \cong C_0(U) \otimes A$ and the diagram
\[
\xymatrix{
\K_0(C_0(\bar{\pi}^{-1}(U))) \ar[r]^{\theta_*} \ar[d]^{\iota_*} & \K_0(C_0(U) \otimes A) \ar[d] ^{\iota_*} \\
\K_0(C_0(\bar{M}_\infty)) \ar[r]^{\theta_*} & \K_0(C_0(M_\infty ,\cA))
}
\]
commutes, where the vertical maps $\iota_*$ are induced from open embeddings.  
By the construction of $\theta_*$, we have $\theta_* \beta_{\bar{U}} = \beta_U \otimes [p] \in \K_0(C_0(U) \otimes A)$, where $p \in \bK(\ell^2(F)) \subset A$ is a rank $1$ projection,. 
Therefore we obtain that
\begin{align*}
 \langle \theta_* \beta _{\bar{M}}, [M,N] \rangle &= \langle  \theta_* \iota_* \beta_{\bar{U}}, [M,N] \rangle
 =\langle \iota_* \theta_* \beta_{\bar{U}}, [M,N] \rangle
 = \langle \theta_* \beta_{\bar{U}}, [U] \rangle\\
 & =  \langle \beta \otimes [p] , [U] \rangle = [p] \in \K_0(A), 
 \end{align*}
and hence
\[ \langle \theta_* [\bv , \mathbf{1} , \mathbf{1}] , [M,N] \rangle = \deg (f_\varepsilon) \langle \theta_* \beta, [M,N] \rangle 
= \deg (f_\varepsilon)\cdot [p]. \]
This finishes the proof since $\K_0(\bK(\ell^2(F))) \to \K_0(A)$ is injective (we recall that the exact sequence (\ref{form:Asplit}) splits). 
\end{proof}

\section{Relative quasi-representations and almost monodromy correspondence}\label{section:5}
Let $\Gamma$ be a countable discrete group and let $\cG$ be a finite subset of $\Gamma$. 
Recall that a map $\pi \colon \Gamma \to \mathrm{U}(P)$ is a \emph{$(\varepsilon , \cG)$-representation} of $\Gamma$ on $P$ if $\pi(e)=1$ and
\[ \|\pi(g)\pi(h)-\pi(gh)\|<\varepsilon   \]
for any $g,h \in \cG$. Let $\qRep_{P}^{\varepsilon , \cG}(\Gamma)$ denote the set of $(\varepsilon, \cG)$-representations of $\Gamma$ on $P$.
\begin{defn}
Let $\pi_1$ and $\pi_2$ be $(\varepsilon , \cG)$-representations of $\Gamma$. An \emph{$\varepsilon$-intertwiner} $u \in \Hom _\varepsilon ( \pi_1 , \pi_2)$ is a unitary $u \in \mathrm{U}(P)$ such that $\| u\pi_1(\gamma )u^* - \pi_2(\gamma) \|< \varepsilon$.
\end{defn}

Let $\phi \colon \Lambda \to \Gamma $ be a homomorphism between countable discrete groups. Let $\cG=(\cG_\Gamma, \cG_\Lambda) $ be a pair of finite subsets $\cG_\Gamma \subset \Gamma $ and $\cG_\Lambda \subset \Lambda$ such that $\phi (\cG_\Lambda ) \subset \cG_\Gamma $. 
\begin{defn}
A \emph{stably relative $(\varepsilon , \cG )$-representation} of $(\Gamma, \Lambda) $ is a quadruple $\boldsymbol{\pi}:=(\pi _1, \pi_2, \pi_0 , u)$, where 
\begin{itemize}
\item $\pi_1 \colon \Gamma \to \mathrm{U}(P)$ and $\pi_2 \colon \Gamma \to \mathrm{U}(P)$ are $(\varepsilon , \cG _\Gamma) $-representations of $\Gamma$,
\item $\pi_0 \colon \Lambda \to \mathrm{U}(Q)$ is a $(\varepsilon , \cG_\Lambda )$-representation of $\Lambda$, and
\item $u \in \Hom _\varepsilon (\pi_1 \circ \phi \oplus \pi_0 ,  \pi_2 \circ \phi \oplus \pi_0)$.
\end{itemize}
We write $\qRep _{P,Q}^{\varepsilon , \cG} (\Gamma , \Lambda )$ for the set of stably relative $( \varepsilon , \cG)$-representations of $(\Gamma, \Lambda )$ on $(P,Q)$. 
\end{defn}
We say that two $(\varepsilon ,\cG)$-representations $\boldsymbol{\pi}$ and $\boldsymbol{\pi}'$ are unitary equivalent if there are unitaries $U_1, U_2 \in \mathrm{U}(P)$ and $U_0 \in \mathrm{U}(Q)$ such that $\pi_i = \Ad (U_i) \circ \pi_i'$ for $i=0,1,2$ and $u'(U_1 \oplus U_0)=(U_2 \oplus U_0)u$.

\begin{rmk}\label{rmk:unstable}
There is an obvious one-to-one correspondence between $\qRep_P^{\varepsilon , \cG} (\Gamma, \Lambda)$ and $\qRep^{\varepsilon , \cG}_P (\Gamma, \phi(\Lambda) )$. Moreover, any relative $(\varepsilon ,\cG)$-representation $(\pi_1 , \pi_2 , u )$ is unitary equivalent to $(\pi _1, \Ad (u^*) \circ \pi_2 , 1)$. That is, up to unitary equivalence we may assume that $u=1$. 
\end{rmk}

Finally we give the almost monodromy correspondence between almost flat bundles on a pair of finite CW-complexes and quasi-representations of the fundamental groups. 

Let $(X,Y)$ be a pair of finite CW-complexes with a good open cover $\cU$. We write $\Gamma := \pi_1(X)$, $\Lambda :=\pi_1(Y)$ and $\phi \colon \Lambda \to \Gamma$ for the map induced from the inclusion.  Fix a maximal subtree $T$ of the $1$-skeleton $N_\cU^{(1)}$ of the nerve of $\cU$ such that $T \cap N_{\cU|_Y}^{(1)}$ is also a maximal subtree of $N_{\cU|_Y}^{(1)}$.
\begin{defn}
We say that $\bv \in \Bdl^{\varepsilon , \cU} _P (X)$ is \emph{normalized on $T$} if $\| v_{\mu\nu}(x)-1 \| < \varepsilon$ for any $\langle \mu,\nu \rangle \in T$. We also says that $\fv=(\bv_1,\bv_2, \bv_0, \bu) \in \Bdl_{P,Q}^{\varepsilon , \cU}(X,Y)$ is normalized on $T$ if $\bv_1$, $\bv_2$ and $\bv_0$ are normalized on $T$. Let $\Bdl_{P}^{\varepsilon , \cU}(X)_T$ (resp.\ $\Bdl_{P,Q}^{\varepsilon , \cU}(X,Y)_T$) denote the set of $(\varepsilon , \cU)$-flat bundles normalized on $T$.
\end{defn}
\begin{lem}\label{lem:normalized}
Any stably relative $(\varepsilon , \cU)$-flat bundle $\fv$ is unitary equivalent (in the sense of Remark \ref{rmk:bdlequiv}) to a stably relative $(\varepsilon , \cU)$-flat bundle normalized on $T$.
\end{lem}
\begin{proof}
It suffices to show that, for any $\bv \in \Bdl_P^{\varepsilon , \cU}(X)$, there is $\bu \in \mathrm{U}(P)^I$ such that $\bu \cdot \bv$ is normalized on $T$.  
Such $\bu$ is constructed inductively (indeed, an inductive construction gives a family $\bu=\{ u_\mu \}_{\mu \in I}$ with the property that $u_\mu = u_{\nu}v_{\mu\nu}(x_{\mu \nu})^*$ for any $\langle \mu , \nu \rangle \in T$).  
\end{proof}
Now we give a one-to-one correspondence up to small correction between (resp.\ stably) relative quasi-representations and almost flat (resp.\ stably) relative bundles normalized on $T$. 

As in Lemma \ref{lem:loctriv}, a $1$-cell $\langle \mu , \nu \rangle \in N_\cU^{(1)} \setminus T$ corresponds to an element $\gamma_{\mu \nu} := [\ell_\mu^{-1} \circ \langle \mu , \nu \rangle \circ \ell_{\nu}]$ of $\Gamma$. Let 
\[ \cG _\Gamma := \{ \gamma _{\mu \nu} \mid \langle \mu , \nu \rangle \in N_\cU^{(1)} \setminus T \} \subset \Gamma . \]
Similarly we define $\cG_\Lambda$ as the set of elements of $\Lambda$ of the form $\gamma_{\mu \nu}$ for $\langle \mu , \nu\rangle \in N_{\cU|_Y}^{(1)} \setminus T$.
Let $\bF_\cG$ denote the free group with the generator $\{s_{\mu \nu} \mid \langle \mu , \nu \rangle \in N_{\cU}^{(1)} \setminus T \} $. We fix a set theoretic section $\tau \colon \Gamma \to \bF_\cG$, that is, $\tau(\gamma _{\mu \nu}) = s_{\mu \nu}$.

\begin{defn}[{\cite[Definition 4.2]{mathOA150306170}}]\label{defn:alpha}
For $\bv \in \Bdl_P^{\varepsilon , \cU}(X)_T$, let
\[\alpha(\bv)(\gamma ):= \prod_{k=1}^n u_{\mu_{k+1}} v_{\mu_{k+1} \mu_{k} } (x_{\mu_{k+1} \mu_{k}}) u_{\mu_k}^* \]
for $\gamma \in \Gamma$ such that $\tau(\gamma) = s_{\mu_1 , \mu_2 } \cdot \dots \cdot s_{\mu_{k-1}, \mu_k} $. Here $u_{\mu}$ is as in (\ref{form:path}).
\end{defn}
It is essentially proved in \cite[Proposition 4.8]{mathOA150306170} that there is a constant $C_{9}=C_9(\cU)$ depending only on $\cU$ such that $\alpha(\bv)$ is a $(C_{9}\varepsilon, \cG)$-representation of $\Gamma $ in $P$.

Conversely, suppose that we have a $(\varepsilon, \cG)$-representation of $\Gamma$. Let $\{ \eta_\mu \} _{\mu \in I}$ and $\{ e_\mu\}_{\mu \in I}$ be as in Remark \ref{rmk:bundle}. Let us define
\begin{align*} 
\breve{\psi}_\mu^{\pi} &:= \sum \eta_\nu \otimes \pi(\gamma_{\nu\mu}) \otimes e_{\nu} \in C(X) \otimes \bB(P) \otimes \bC^I, \\
 v_{\mu\nu}^\pi &:= ((\breve{\psi}_\nu^{\pi})^*\breve{\psi}_\nu^{\pi})^{-1/2}((\breve{\psi}_\nu^{\pi})^*\breve{\psi}_\mu^{\pi})((\breve{\psi}_\mu^{\pi})^*\breve{\psi}_\mu^{\pi})^{-1/2}.
\end{align*}
By Lemma \ref{lem:polar}, we have the inequality $\| v_{\mu \nu}^\pi(x) - \pi (\gamma_{\mu \nu } ) \| <4 \varepsilon$. This implies that $\bv^\pi := \{ v_{\mu\nu}^\pi \}_{\mu,\nu \in I}$ is $(8\varepsilon, \cU)$-flat bundle normalized on $T$.
\begin{defn}\label{defn:beta}
For $\pi \in \qRep_P^{\varepsilon , \cG }(\Gamma)$, we define $\beta(\pi)$ to be $\bv^\pi \in \Bdl_{P}^{8\varepsilon, \cU} (X)_T$.
\end{defn}

We consider the distance in $\Bdl^{\varepsilon , \cU}_P(X)$ and $\qRep^{\varepsilon, \cG}_P(\Gamma)$ defined as
\begin{align*}
d(\bv, \bv')&:= \sup_{\mu , \nu \in I} \| v_{\mu\nu} - v'_{\mu \nu} \|, \\
d(\pi , \pi')&:=  \sup _{\gamma \in \cG_\Gamma} \| \pi(\gamma) - \pi'(\gamma) \|.
\end{align*}
\begin{lem}\label{lem:alpha}
There is a constant $C_{10}=C_{10}(\cU)>0$ depending only on $\cU$ such that the maps $\alpha $ and $\beta$ satisfy
\begin{align*}
\begin{split}
d(\alpha (\bv), \alpha (\bv ' )) &\leq d(\bv, \bv') + C_{10}\varepsilon, \\
d(\beta (\pi), \beta (\pi ')) &\leq d(\pi , \pi' ) + C_{10}\varepsilon, \\
d(\beta \circ \alpha  (\bv) , \bv) &\leq C_{10}\varepsilon, \\
d(\alpha \circ \beta (\pi), \pi) &\leq C_{10}\varepsilon.
\end{split}
\end{align*}
\end{lem}
\begin{proof}
By Corollary \ref{cor:pull}, we may assume that $X$ is a finite simplicial complex and $\cU$ is the open cover of $X$ consisting of star neighborhoods $U_\mu$ of $0$-cells $\mu$. We choose $x_{\mu\nu}$ as the median of the $1$-cell $\langle \mu,\nu \rangle$.

Let $\mathrm{GL}(P)_\delta$ denote the set of $T \in \bB(P)$ with $d(T,\mathrm{U}(P))<\varepsilon$ and let $\mathrm{Crd}_P^{\varepsilon}(X)_T$ denote the set of $\varepsilon$-flat coordinate bundles on $X$ normalized on $T$. Here, an $\varepsilon$-flat coordinate bundle on a simplicial complex is a family $\{ v_{\mu \nu}\}$ of $\varepsilon$-flat $\mathrm{GL}(P)_\varepsilon$-valued functions $v_{\mu\nu}$ on the union of simplices of the barycentric subdivision of $X$ included to $U_\mu \cap U_\nu$ which satisfies the cocycle relation (for the precise definition, see \cite[Definition 2.5]{mathOA150306170}). It is said to be normalized on $T$ if $v_{\mu\nu}(x_{\mu\nu})=1$ for $\langle \mu,\nu \rangle \in T$. We remark that the restriction gives a map $\cR \colon \mathrm{Bdl}_P^{\varepsilon , \cU}(X)_T \to \mathrm{Crd}_P^\varepsilon (X)_T$.

Let $\mathop{\overline{\mathrm{qRep}}}_P^{\varepsilon, \cG}(\Gamma )$ denote the set of $(\varepsilon, \cG)$-representation which takes value in $\mathrm{GL}(P)_\varepsilon$ instead of $\mathrm{U}(P)$. In \cite{mathOA150306170}, Carri\'{o}n and Dadarlat construct maps 
\[ \alpha_{\mathrm{CD}}  \colon \mathrm{Crd}_P^\varepsilon (X) \to \mathop{\overline{\mathrm{qRep}}}\nolimits _P^{C_{10}'\varepsilon ,\cG}(\Gamma)_T , \ \ \ \ \beta_{\mathrm{CD}} \colon \mathop{\overline{\mathrm{qRep}}}\nolimits_P^{\varepsilon, \cG} (\Gamma) \to \mathrm{Crd}_{P}^{C_{10}'\varepsilon}(X),\] 
which is compatible with our $\alpha$ and $\beta$ in the sense that
\begin{itemize}
\item $d(\bv,\bv')-2\varepsilon \leq d(\cR (\bv) , \cR(\bv')) \leq d(\bv,\bv')$ for any $\bv,\bv' \in \Bdl^{\varepsilon, \cU}_P(X)$,
\item $\alpha _{\mathrm{CD}} \circ \cR(\bv) = \alpha(\bv)$ for any $\bv \in \Bdl_{P}^{\varepsilon , \cU}(X)$,
\item $d(\cR \circ \beta(\pi) , \beta_{\mathrm{CD}}(\pi))< (C_{10}' + 8)\varepsilon$ for $\pi \in \qRep_{P}^{\varepsilon, \cG}(\Gamma)$.
\end{itemize}
Here, the second is obvious from the constructions (compare \cite[Definition 4.2]{mathOA150306170} with Definition \ref{defn:alpha}) and the third follows from $\beta_{\mathrm{CD}}(\pi)_{\mu\nu}(x_{\mu\nu}) = \pi(\gamma_{\mu\nu}) $ (which is obvious from the construction \cite[Definition 5.3]{mathOA150306170}) and the inequality $ \| v^\pi_{\mu\nu}(x) -\pi(\gamma _{\mu\nu}) \| < 4 \varepsilon$ remarked above. Now, the lemma follows from \cite[Theorem 3.1, Theorem 3.3]{mathOA150306170}.
\end{proof}

\begin{lem}\label{lem:ab_rel}
Let $\Delta_I \colon \mathrm{U}(P) \to \mathrm{U}(P)^I$ denote the diagonal embedding. There is a constant $C_{11}=C_{11}(\cU)$ depending only on $\cU$ such that the following hods: 
\begin{enumerate}
\item Let $\pi_1 , \pi_2 \in \qRep_P^{\varepsilon, \cG } (\Lambda)$. If there exists $u \in \Hom_\varepsilon (\pi_1 , \pi_2)$, then $\Delta _I(u) \in \mathrm{U}(P)$ is contained in $\Hom_{C_{11}\varepsilon} (\beta(\pi_1), \beta(\pi_2))$.
\item Let $\bv_1, \bv_2 \in \Bdl_P^{\varepsilon , \cU|_Y} (Y)_T$. If there exists $\bu \in \Hom _\varepsilon (\bv_1, \bv_2)$, then $\| u_\mu - u_\nu \| \leq C_{11} \varepsilon$ and $u_\mu \in \Hom_{C_{11}\varepsilon} (\alpha (\bv_1) , \alpha(\bv_2)) $. 
\end{enumerate}
\end{lem}
\begin{proof}
To see (1), let $\bv_i:=\beta (\pi_i)$. By Lemma \ref{lem:alpha}, we have 
\[  d(\bv_1, \bu \cdot \bv_2) = d(\beta (\pi_1) , \beta (\Ad (u) \circ \pi_2)) < C_8\varepsilon + d(\pi_1, \Ad (u) \circ \pi_2) = (C_{10}+1)\varepsilon.  \]
This means that  $\Delta_I(u) \in \Hom_{(C_{10}+1)\varepsilon }(\bv_1 , \bv_2)$. 

Next we show (2). If $\langle \mu , \nu \rangle \in T$, we get
\[ \| u_\mu - u_\nu \| \leq \| u_\mu v_{\mu \nu}^1 (x_{\mu \nu}) u_\nu^* - 1 \| + \| u_{\mu}(v_{\mu \nu}^1(x_{\mu \nu}) -1)\| < 2\varepsilon \]
and hence $\| u_\mu -u_\nu \|<2 \mathop{\mathrm{diam}}(T)\varepsilon$. Therefore we have 
\begin{align*}
d(\pi_1, \Ad (u_\mu) \circ \pi_2 ) & = d(\alpha (\bv_1), \alpha (\Delta_I(u_\mu) \cdot \bv_2)) \\
& \leq d(\bv_1, \bu \cdot \bv_2) + d(\bu \cdot \bv_2 , \Delta _I(u_\mu) \cdot \bu_2) + C_{10} \varepsilon\\
&  \leq  (1+ 2\mathop{\mathrm{diam}}(T) +C_{10})\varepsilon. 
\end{align*}
Now the proof is completed by choosing $C_{11}:=C_{10}+ 1+2 \mathop{\mathrm{diam}}(T)$. 
\end{proof}

\begin{defn}

Let us fix $\mu_0 \in I$ and let $C_{12}= \max \{ C_9 , 8 , C_{11} \}$. We define two maps 
\begin{align*}
\boldsymbol{\alpha}& \colon \Bdl_{P,Q}^{\varepsilon ,\cU} (X,Y)_T \to \qRep ^{C_{12}\varepsilon, \cG}_{P,Q}(\Gamma, \Lambda ), \\
\boldsymbol{\beta}& \colon \qRep^{ \varepsilon, \cG}_{P,Q} (\Gamma, \Lambda) \to \Bdl_{P,Q}^{C_{12}\varepsilon, \cU} (X,Y)_T,
\end{align*}
by
\begin{align*}
\boldsymbol{\alpha}(\bv_1, \bv_2, \bv_0, \bu )&= (\alpha (\bv_1), \alpha(\bv_2), \alpha (\bv_0), u_{\mu_0} ),\\
\boldsymbol{\beta}(\pi_1,\pi_2,\pi_0, u) &= (\beta(\pi_1), \beta(\pi_2), \beta(\pi_0), \Delta_I(u)).
\end{align*}
\end{defn}

We define the metric on $\Bdl_{P,Q}^{\varepsilon , \cU} (X,Y)$ and $\qRep_{P,Q}^{\varepsilon , \cG} (\Gamma, \Lambda)$ by 
\begin{align*}
d(\fv,\fv') &:= \max \{ d(\bv_1, \bv_1'), d(\bv_2, \bv_2'), d(\bv_0, \bv_0'), d(\bu,\bu') \} \\
d(\boldsymbol{\pi} , \boldsymbol{\pi} ' ) &:= \max \{d(\pi_1, \pi_1'), d(\pi_2, \pi_2'), d(\pi_0, \pi_0'), d(u,u')\}.
\end{align*}
\begin{lem}\label{lem:homotopy}
If $\fv , \fv' \in \Bdl_{P,Q}^{\varepsilon , \cU} (X,Y)$ satisfies $d(\fv,\fv') < \varepsilon $, then $\fv_1$ and $\fv_2$ are homotopic in the space $\Bdl_{P,Q}^{(4C_1 +1) \varepsilon, \cU} (X,Y)$.
\end{lem}
\begin{proof}
By Lemma \ref{lem:gauge}, there are $\{ \bar{u}_\mu^i \}$ for $i=1,2,0$ such that $\bar{u}^i_\mu v_{\mu\nu}^i (\bar{u}_\nu^i)^* = (v')^i_{\mu \nu}$ and $\| \bar{u}_\mu -1 \| < C_1 \varepsilon$. Since $\bar{u}$ is near to the identity, $\bar{u}^{i,s}_\mu:= \exp (s \log (\bar{u}^i_\mu)) $ is a unitary-valued functions such that $\| \bar{u}^{i,s}_\mu -1\| < C_1 \varepsilon$. Then
\[ (\{\bar{u}^{1,s}_\mu v^1_{\mu \nu}(\bar{u}_\nu^{1,s})^* \}_{\mu,\nu} , \{\bar{u}^{2,s}_\mu v^2_{\mu \nu}(\bar{u}_\nu^{2,s})^* \}_{\mu,\nu}, \{\bar{u}^{0,s}_\mu v^0_{\mu \nu}(\bar{u}_\nu^{0,s})^* \}_{\mu,\nu}, u) \]
is a continuous path in $\Bdl^{(4C_1 + 1)\varepsilon , \cU}(X,Y)$ connecting $\fv$ with $(\bv_1',\bv_2',\bv_0', \bu)$. Also, $\bu_s = \{u_\mu^s:=  u_\mu \exp (s \log (u_\mu^* u_\mu')) \}_{\mu \in I}$ is a continuous path connecting $\bu$ with $\bu'$ such that $\| u^s_\mu - u_\mu' \| < \varepsilon$, which makes $(\bv_1',\bv_2',\bv_0',\bu_s)$ to a homotopy of $(3\varepsilon , \cU)$-flat bundles. 
\end{proof}

\begin{thm}\label{prp:rel1to1}
Let $(X,Y)$ be a finite simplicial complex and let $\Gamma := \pi_1(X)$ and $\Lambda :=\pi_1(Y)$.
\begin{enumerate}
\item For $\fv, \fv' \in \Bdl^{\varepsilon,\cU}_{P,Q} (X,Y)_T$, we have $d(\boldsymbol{\alpha} (\fv) , \boldsymbol{\alpha}(\fv')) \leq d(\fv , \fv') + C_{10}\varepsilon$ and $d(\boldsymbol{\beta} \circ  \boldsymbol{\alpha}(\fv ) , \fv ) \leq C_{11}\varepsilon$.
\item For $\boldsymbol{\pi}, \boldsymbol{\pi}' \in \qRep^{\varepsilon, \cG}_{P,Q} (\Gamma, \Lambda)$, we have $d(\boldsymbol{\beta}(\boldsymbol{\pi}), \boldsymbol{\beta} ( \boldsymbol{\pi}')) \leq d(\boldsymbol{\pi}, \boldsymbol{\pi}') + C_{10}\varepsilon$ and $d(\boldsymbol{\alpha} \circ \boldsymbol{\beta }(\boldsymbol{\pi}) , \boldsymbol{\pi} )\leq C_{11}\varepsilon$.
\end{enumerate}
\end{thm}
\begin{proof}
It follows from Lemma \ref{lem:alpha} and Lemma \ref{lem:ab_rel}.
\end{proof}

\begin{cor}
If there is a continuous map $f \colon (X_1,Y_1) \to (X_2, Y_2)$ which induces the isomorphism of fundamental groups, then $\K_{\mathrm{s\mathchar`-af}}^0(X_1,Y_1;A)$ is included to $f^* \K_{\mathrm{s\mathchar`-af}}^0(X_2,Y_2;A)$. In particular, if $(B\Gamma, B\Lambda)$ has the homotopy type of a pair of finite CW-complexes, then $\K_{\mathrm{s\mathchar`-af}}^0(X,Y;A) \subset f^*\K^0(B\Gamma, B\Lambda ;A)$, where $f$ is the reference map.
\end{cor}
\begin{proof}
For sufficiently small $\varepsilon >0$, let $\fv \in \Bdl ^{\varepsilon, \cU_1}_{P,Q}(X_1,Y_1)$ be a representative of $\xi \in \K_{\mathrm{s\mathchar`-af}}^0(X_1,Y_1;A)$.
By Remark \ref{rmk:bdlequiv} and Lemma \ref{lem:normalized}, we may assume without loss of generality that $\fv$ is normalized on $T$.
Here we write $\boldsymbol{\alpha}_{X,Y}$ and $\boldsymbol{\beta}_{X,Y}$ for the map $\boldsymbol{\alpha}$ and $\boldsymbol{\beta} $ with respect to the pair $(X,Y)$.
Then, $\tilde{\fv}:= \boldsymbol{\beta} _{X_2, Y_2} \circ \boldsymbol{\alpha} _{X_1,Y_1}(\fv)$ is a $(C_{12}(\cU_1) C_{12}(\cU_2) \varepsilon , \cU_2)$-flat bundle on $(X_2,Y_2)$ which satisfies $d(\fv, f^*\tilde{\fv})<C_{11}(\cU)\varepsilon$.  Hence $[\fv] = f^*[\tilde{\fv}]$ by Lemma \ref{lem:homotopy}. 
\end{proof}

\begin{rmk}\label{rmk:double}
Let $(X,Y)$ be a pair of finite CW-complexes with $\pi_1(X):=\Gamma $ and $\pi_1(Y) :=\Lambda$ and let $\cU$ be an open cover of $(X,Y)$. Assume that the induced map $\Lambda \to \Gamma $ is injective. Then the double $\hat{X}:=X \sqcup_Y Y \times [0,1] \sqcup_Y X$ has the fundamental group $\Gamma \ast_\Lambda \Gamma $ by the van Kampen theorem. 
We associate an open cover $\hat{\cU}$ of $\hat{X}$ to $\cU$ as 
\[ \hat{\cU} = \{ U_{\mu, i} :=q_i^* U_{\mu}  \cap X_i^\circ \}_{(\mu, i) \in I \times \{1,2\}}, \]
 where $X_1 : = X \sqcup Y \times [0,1]$, $X_2:=Y \times [0,1] \sqcup X$ and $q_i \colon X_i  \to X$ for $i=1,2$ are canonical retractions. Let $\hat{\cG}\subset \Gamma \ast_\Lambda \Gamma$ denote the union of two copies of $\cG_\Gamma \subset \Gamma$.

In this setting, there is a correspondence
\[\xymatrix{
\Bdl _P^{\varepsilon , \cU} (X,Y) \ar@<0.5ex>[r] \ar@<0.5ex>[d]^{\boldsymbol{\alpha}} & \Bdl _P ^{\varepsilon , \hat{\cU}} (\hat{X}) \ar@<0.5ex>[d]^\alpha \ar@<0.5ex>[l] \\
\qRep_P^{\varepsilon , \cG} (\Gamma, \Lambda) \ar@<0.5ex>[r] \ar@<0.5ex>[u]^{\boldsymbol{\beta}}& \qRep_P^{\varepsilon , \hat{\cG}} (\Gamma \ast _\Lambda \Gamma ), \ar@<0.5ex>[u]^\beta \ar@<0.5ex>[l]
}\]
which commutes up to small perturbations. This is a counterpart in almost flat geometry of the higher index theory of invertible doubles studied in \cite[Section 5]{Kubota1}.

\begin{itemize}
\item We fix a point $x_{\mu \nu} \in U_{\mu \nu} \cap Y$ for each $\mu , \nu \in I$ with $U_{\mu \nu} \cap Y \neq \emptyset$. For $\hat{\bv} \in \Bdl _P^{\varepsilon , \hat{\cU}}(\hat{X})$, let $\bv_i:=\{ \hat{v}_{(\mu , i) (\nu , i)}|_{q_i^*U_{\mu \nu} \cap X} \}_{\mu , \nu \in I}$ for $i=1,2$ and $\bu:=\{ u_\mu := \hat{v}_{(\mu , 1) (\mu , 2)}(x_{\mu \nu}) \}$ for $\mu , \nu \in I$ with $U_{\mu \nu} \cap Y \neq \emptyset$. Then $(\bv_1 , \bv_2 , \bu)$ is a relative $(\varepsilon , \cU)$-flat bundle on $(X,Y)$.
\item For $\fv=(\bv_1, \bv_2 , \bu) \in \Bdl_{P}^{\varepsilon , \cU}(X,Y)$, pick $\bar{u} \in \cG_{C_1\varepsilon} (\bu)$ by Lemma \ref{lem:gauge}. Then $\hat{\bv}=\{ \hat{v}_{(\mu ,i)  (\nu , j)} \}$ given by 
\[ \hat{v}_{(\mu ,i)  (\nu , j) }:=
\left\{ \begin{array}{ll} 
(q_1^* v_{ \mu \nu}^i)|_{q_1^*U_{\mu\nu} \cap X_i^\circ} & \text{ if $i=j$,} \\
q_1^*(v_{\mu \nu }^2 \bar{u}_\nu)|_{q_1^*U_{\mu \nu} \cap Y \times (0,1)} & \text{ if $i=1$, $j=2$}, \\ 
q_1^*(v_{\mu \nu }^1 \bar{u}_\nu^*)|_{q_1^*U_{\mu \nu} \cap Y \times (0,1)} & \text{ if $i=2$, $j=1$},  \end{array} \right. \]
is a $((C_1+1)\varepsilon , \hat\cU)$-flat bundle on $\hat{X}$. 
\item For a $(\varepsilon , \hat{\cG} )$-representation $\hat{\pi}$ of $\Gamma \ast_{\Lambda} \Gamma$, let $\pi_1$ and $\pi_2$ denote its restrictions to the first and second copies of $\Gamma$. Then, $\pi \mapsto (\pi_1 , \pi_2 , 1)$ gives a map from $\qRep_{P}^{\varepsilon, \hat{\cG} } (\Gamma \ast_{\Lambda} \Gamma )$ to $\qRep_{P}^{\varepsilon , \cG}(\Gamma, \Lambda )$. 
\item For $\boldsymbol{\pi}  \in \qRep_{P}^{\varepsilon , \cG}(\Gamma, \Lambda )$ of the form $(\pi_1, \pi _2  , 1)$, a $(2\varepsilon ,\hat{\cG})$-representation $\hat{\pi}$ of $\Gamma \ast_\Lambda \Gamma$ constructed in the following way. Pick a set theoretic section $\tau \colon \Gamma \ast_{\Lambda} \Gamma \to \Gamma \ast \Gamma $ and let $\hat{\pi}(\gamma ):= (\pi_1 \ast \pi_2)(\tau(\gamma ))$. 
Then $\hat{\pi}$ is a $(2\varepsilon , \hat{\cG})$-representation of $\Gamma \ast_\Lambda \Gamma$.\end{itemize}
\end{rmk}

\bibliographystyle{alpha}
\bibliography{bibABC,bibDEFG,bibHIJK,bibLMN,bibOPQR,bibSTUV,bibWXYZ,arXiv}

\def\cprime{$'$} \def\cprime{$'$} \def\cdprime{$''$} \def\cprime{$'$}
  \def\cprime{$'$} \def\ocirc#1{\ifmmode\setbox0=\hbox{$#1$}\dimen0=\ht0
  \advance\dimen0 by1pt\rlap{\hbox to\wd0{\hss\raise\dimen0
  \hbox{\hskip.2em$\scriptscriptstyle\circ$}\hss}}#1\else {\accent"17 #1}\fi}
  \def\cprime{$'$} \def\cdprime{$''$} \def\cprime{$'$} \def\cprime{$'$}
  \def\cprime{$'$} \def\cprime{$'$} \def\cprime{$'$} \def\cprime{$'$}
  \def\cprime{$'$} \def\cprime{$'$} \def\cprime{$'$} \def\cprime{$'$}
  \def\cprime{$'$} \def\cprime{$'$} \def\Dbar{\leavevmode\lower.6ex\hbox to
  0pt{\hskip-.23ex \accent"16\hss}D}
  \def\cftil#1{\ifmmode\setbox7\hbox{$\accent"5E#1$}\else
  \setbox7\hbox{\accent"5E#1}\penalty 10000\relax\fi\raise 1\ht7
  \hbox{\lower1.15ex\hbox to 1\wd7{\hss\accent"7E\hss}}\penalty 10000
  \hskip-1\wd7\penalty 10000\box7}
  \def\cfudot#1{\ifmmode\setbox7\hbox{$\accent"5E#1$}\else
  \setbox7\hbox{\accent"5E#1}\penalty 10000\relax\fi\raise 1\ht7
  \hbox{\raise.1ex\hbox to 1\wd7{\hss.\hss}}\penalty 10000 \hskip-1\wd7\penalty
  10000\box7} \def\cprime{$'$} \def\cprime{$'$} \def\cdprime{$''$}
  \def\polhk#1{\setbox0=\hbox{#1}{\ooalign{\hidewidth
  \lower1.5ex\hbox{`}\hidewidth\crcr\unhbox0}}} \def\cprime{$'$}
  \def\cprime{$'$} \def\cprime{$'$}
  \def\cftil#1{\ifmmode\setbox7\hbox{$\accent"5E#1$}\else
  \setbox7\hbox{\accent"5E#1}\penalty 10000\relax\fi\raise 1\ht7
  \hbox{\lower1.15ex\hbox to 1\wd7{\hss\accent"7E\hss}}\penalty 10000
  \hskip-1\wd7\penalty 10000\box7} \def\cprime{$'$}
  \def\polhk#1{\setbox0=\hbox{#1}{\ooalign{\hidewidth
  \lower1.5ex\hbox{`}\hidewidth\crcr\unhbox0}}} \def\cprime{$'$}
  \def\cprime{$'$} \def\cprime{$'$} \def\cprime{$'$} \def\cprime{$'$}
  \def\cprime{$'$} \def\cprime{$'$} \def\cprime{$'$} \def\cprime{$'$}
  \def\cprime{$'$} \def\cprime{$'$} \def\cprime{$'$} \def\cprime{$'$}
  \def\polhk#1{\setbox0=\hbox{#1}{\ooalign{\hidewidth
  \lower1.5ex\hbox{`}\hidewidth\crcr\unhbox0}}}
  \def\polhk#1{\setbox0=\hbox{#1}{\ooalign{\hidewidth
  \lower1.5ex\hbox{`}\hidewidth\crcr\unhbox0}}} \def\cprime{$'$}
  \def\cprime{$'$} \def\cprime{$'$}
  \def\polhk#1{\setbox0=\hbox{#1}{\ooalign{\hidewidth
  \lower1.5ex\hbox{`}\hidewidth\crcr\unhbox0}}}
  \def\polhk#1{\setbox0=\hbox{#1}{\ooalign{\hidewidth
  \lower1.5ex\hbox{`}\hidewidth\crcr\unhbox0}}}
  \def\polhk#1{\setbox0=\hbox{#1}{\ooalign{\hidewidth
  \lower1.5ex\hbox{`}\hidewidth\crcr\unhbox0}}}
  \def\polhk#1{\setbox0=\hbox{#1}{\ooalign{\hidewidth
  \lower1.5ex\hbox{`}\hidewidth\crcr\unhbox0}}} \def\cprime{$'$}
  \def\cprime{$'$} \def\cprime{$'$} \def\cprime{$'$} \def\cprime{$'$}
  \def\cprime{$'$} \def\cprime{$'$}
\begin{bibdiv}
\begin{biblist}

\bib{mathOA150306170}{article}{
      author={Carri{\'o}n, Jos{\'e}~R.},
      author={Dadarlat, Marius},
       title={{A}lmost flat {K}-theory of classifying spaces},
        date={2018},
     journal={J. Noncommut. Geom.},
      volume={12},
      number={2},
}

\bib{MR1042862}{article}{
      author={Connes, Alain},
      author={Gromov, Mikha{\"{\i}}l},
      author={Moscovici, Henri},
       title={Conjecture de {N}ovikov et fibr\'es presque plats},
        date={1990},
        ISSN={0764-4442},
     journal={C. R. Acad. Sci. Paris S\'er. I Math.},
      volume={310},
      number={5},
       pages={273\ndash 277},
      review={\MR{1042862 (91e:57041)}},
}

\bib{MR1065438}{article}{
      author={Connes, Alain},
      author={Higson, Nigel},
       title={D\'eformations, morphismes asymptotiques et {$K$}-th\'eorie
  bivariante},
        date={1990},
        ISSN={0764-4442},
     journal={C. R. Acad. Sci. Paris S\'er. I Math.},
      volume={311},
      number={2},
       pages={101\ndash 106},
      review={\MR{1065438 (91m:46114)}},
}

\bib{MR3275029}{article}{
      author={Dadarlat, Marius},
       title={Group quasi-representations and almost flat bundles},
        date={2014},
        ISSN={1661-6952},
     journal={J. Noncommut. Geom.},
      volume={8},
      number={1},
       pages={163\ndash 178},
         url={http://dx.doi.org/10.4171/JNCG/152},
      review={\MR{3275029}},
}

\bib{MR720933}{article}{
      author={Gromov, Mikhael},
      author={Lawson, H.~Blaine, Jr.},
       title={Positive scalar curvature and the {D}irac operator on complete
  {R}iemannian manifolds},
        date={1983},
        ISSN={0073-8301},
     journal={Inst. Hautes \'Etudes Sci. Publ. Math.},
      number={58},
       pages={83\ndash 196 (1984)},
         url={http://www.numdam.org/item?id=PMIHES_1983__58__83_0},
      review={\MR{720933}},
}

\bib{MR1389019}{incollection}{
      author={Gromov, M.},
       title={Positive curvature, macroscopic dimension, spectral gaps and
  higher signatures},
        date={1996},
   booktitle={Functional analysis on the eve of the 21st century, {V}ol.\ {II}
  ({N}ew {B}runswick, {NJ}, 1993)},
      series={Progr. Math.},
      volume={132},
   publisher={Birkh\"auser Boston, Boston, MA},
       pages={1\ndash 213},
         url={https://doi.org/10.1007/s10107-010-0354-x},
      review={\MR{1389019}},
}

\bib{MR3289846}{incollection}{
      author={Hanke, Bernhard},
       title={Positive scalar curvature, {K}-area and essentialness},
        date={2012},
   booktitle={Global differential geometry},
      series={Springer Proc. Math.},
      volume={17},
   publisher={Springer, Heidelberg},
       pages={275\ndash 302},
         url={https://doi.org/10.1007/978-3-642-22842-1_10},
      review={\MR{3289846}},
}

\bib{MR2259056}{article}{
      author={Hanke, B.},
      author={Schick, T.},
       title={Enlargeability and index theory},
        date={2006},
        ISSN={0022-040X},
     journal={J. Differential Geom.},
      volume={74},
      number={2},
       pages={293\ndash 320},
         url={http://projecteuclid.org/euclid.jdg/1175266206},
      review={\MR{2259056}},
}

\bib{MR2353861}{article}{
      author={Hanke, Bernhard},
      author={Schick, Thomas},
       title={Enlargeability and index theory: infinite covers},
        date={2007},
        ISSN={0920-3036},
     journal={$K$-Theory},
      volume={38},
      number={1},
       pages={23\ndash 33},
         url={https://doi.org/10.1007/s10977-007-9004-3},
      review={\MR{2353861}},
}

\bib{mathGT160707820}{misc}{
      author={Hunger, Benedikt},
       title={{A}lmost flat bundles and homological invariance of infinite
  {K}-area},
        date={2016},
        note={preprint, \href{https://arxiv.org/abs/1607.07820}{{\tt
  arXiv:1607.07820[math.GT]}}},
}

\bib{MR2458205}{book}{
      author={Karoubi, Max},
       title={{$K$}-theory},
      series={Classics in Mathematics},
   publisher={Springer-Verlag, Berlin},
        date={2008},
        ISBN={978-3-540-79889-7},
         url={http://dx.doi.org/10.1007/978-3-540-79890-3},
        note={An introduction, Reprint of the 1978 edition, With a new postface
  by the author and a list of errata},
      review={\MR{2458205 (2009i:19001)}},
}

\bib{Kubota1}{misc}{
      author={Kubota, Yosuke},
       title={{T}he relative {M}ishchenko--{F}omenko higher index and almost
  flat bundles {I}},
        date={2018},
        note={preprint, \href{http://arxiv.org/abs/1609.06690}{{\tt
  arXiv:1807.03181[math.KT]}}},
}

\bib{Kubota2}{misc}{
      author={Kubota, Yosuke},
       title={{T}he relative {M}ishchenko--{F}omenko higher index and almost
  flat bundles {II}},
        date={2019},
        note={preprint.},
}

\bib{MR3101798}{article}{
      author={Listing, Mario},
       title={Homology of finite {K}-area},
        date={2013},
        ISSN={0025-5874},
     journal={Math. Z.},
      volume={275},
      number={1-2},
       pages={91\ndash 107},
         url={https://doi.org/10.1007/s00209-012-1124-7},
      review={\MR{3101798}},
}

\bib{MR1865957}{article}{
      author={Manuilov, V.~M.},
      author={Mishchenko, A.~S.},
       title={Almost, asymptotic and {F}redholm representations of discrete
  groups},
        date={2001},
        ISSN={0167-8019},
     journal={Acta Appl. Math.},
      volume={68},
      number={1-3},
       pages={159\ndash 210},
         url={https://doi.org/10.1023/A:1012077015181},
        note={Noncommutative geometry and operator $K$-theory},
      review={\MR{1865957}},
}

\bib{MR2179351}{article}{
      author={Mishchenko, Alexander~S.},
      author={Teleman, Nicolae},
       title={Almost flat bundles and almost flat structures},
        date={2005},
        ISSN={1230-3429},
     journal={Topol. Methods Nonlinear Anal.},
      volume={26},
      number={1},
       pages={75\ndash 87},
         url={https://doi.org/10.12775/TMNA.2005.025},
      review={\MR{2179351}},
}

\end{biblist}
\end{bibdiv}

\end{document}